\documentclass[reqno]{amsart}%
\usepackage{graphicx}
\usepackage{amsmath}
\usepackage{amsfonts}
\usepackage{amssymb}
\usepackage[numbers,square]{natbib}%
\usepackage[hidelinks]{hyperref}
\usepackage{lmodern}
\newtheorem{theorem}{Theorem}

\newtheorem{corollary}[theorem]{Corollary}

\newtheorem{definition}[theorem]{Definition}
\newtheorem{example}[theorem]{Example}

\newtheorem{lemma}[theorem]{Lemma}
\newtheorem{notation}[theorem]{Notation}

\newtheorem{proposition}[theorem]{Proposition}
\newtheorem{remark}[theorem]{Remark}

\numberwithin{theorem}{section}
\numberwithin{equation}{section}

\begin{document}
\title[The heat flow, GAF, and $SL(2;\mathbb{R})$]{The heat flow, GAF, and $SL(2;\mathbb{R})$}
\author{Brian C. Hall}
\address{Brian C. Hall: Department of Mathematics, University of Notre Dame, Notre Dame,
IN 46556, USA}
\email{bhall@nd.edu}
\author{Ching-Wei Ho}
\address{Ching-Wei Ho: Institute of Mathematics, Academia Sinica, Taipei 10617, Taiwan}
\email{chwho@gate.sinica.edu.tw}
\author{Jonas Jalowy}
\address{Jonas Jalowy: Institut f\"ur Mathematische Stochastik, Westf\"alische
Wilhelms-Universit\"at\linebreak M\"unster, Orl\'eans-Ring 10, 48149
M\"unster, Germany}
\email{jjalowy@uni-muenster.de}
\author{Zakhar Kabluchko}
\address{Zakhar Kabluchko: Institut f\"ur Mathematische Stochastik, Westf\"alische
Wilhelms-Universit\"at M\"unster, Orl\'eans-Ring 10, 48149 M\"unster, Germany}
\email{zakhar.kabluchko@uni-muenster.de}
\begin{abstract}
We establish basic properties of the heat flow on entire holomorphic functions
that have order at most 2. We then look specifically at the action of the heat
flow on the Gaussian analytic function (GAF). We show that applying the heat
flow to a GAF and then rescaling and multiplying by an exponential of a quadratic function gives another
GAF. It follows that the zeros of the GAF are invariant in distribution under
the heat flow, up to a simple rescaling.

We then show that the zeros of the GAF evolve under the heat flow
approximately along straight lines, with an error whose distribution is
independent of the starting point. Finally, we connect the heat flow on the
GAF to the metaplectic representation of the double cover of the group
$SL(2;\mathbb{R}).$
\end{abstract}
\maketitle
\tableofcontents
\addtocontents{toc}{\protect\enlargethispage*{\baselineskip}}
\noindent
%\bigskip 
\section{Introduction}
The aim of this note is to study the action of the heat flow operator on some
deterministic and random entire functions. The main example we are interested
in is the (flat) \emph{Gaussian analytic function} (or
GAF)~\cite{hough_etal_book,sodin_tsirelson,sodin_review} defined by
\begin{equation}
G(z)=\sum_{n=0}^{\infty}\xi_{n}\frac{z^{n}}{\sqrt{n!}}, \label{eq:GAFdef}%
\end{equation}
where $\xi_{0},\xi_{1},\ldots$ are i.i.d.\ random variables with the standard complex Gaussian distribution (with Lebesgue density on $\mathbb{C}$ given by $\pi
^{-1}\mathrm{e}^{-|z|^{2}}$). It is known that $G(z)$ is a random entire
function whose set of zeros is a stationary point process, meaning that the
zero set of $G(z-a)$ has the same distribution as the zero set of $G(z)$, for
all $a\in\mathbb{C}$. The finite-dimensional distributions of the stochastic
process $(G(z))_{z\in\mathbb{C}}$ are multivariate complex Gaussian, and the
covariance function is given by
\begin{equation}
\mathbb{E}\left[  G(z)\overline{G(w)}\right]  =\exp\{z\overline{w}\},\qquad
z,w\in\mathbb{C}. \label{eq:GAF_covariance}%
\end{equation}

\subsection{Heat flow on entire functions}

Let $
D=d/dz
$
be the differentiation operator acting on analytic functions of the complex
variable $z$. Given a complex number $\tau\in\mathbb{C}$ and an analytic
function $F(z)$, the \emph{heat-flow operator}\footnote{If $\tau$ is a positive real number and $F$ is restricted to the real line,
$e^{-\tau D^{2}/2}$ is actually what is usually called the backward heat
operator. In the present paper, however, we will allow the time parameter $\tau$ to be an arbitrary complex number, in some
cases subject to a restriction on the size of $\left\vert \tau\right\vert$, so that the choice of the sign is only for consistency with a forthcoming paper. Also note that the heat flow of our consideration is the heat flow in one complex variable and should not be confused with the heat flow in two real variables, that is $\exp\{-\frac \tau 2 \frac{\partial}{\partial z}\frac{\partial}{\partial\bar z}\}$, which is the identity on the space of holomorphic functions.} $\exp\{-\tau D^{2}/2\}$ is
defined by the formal series
\begin{equation}
F(z;\tau):=e^{-\tau D^{2}/2}F(z)=\sum_{k=0}^{\infty}\frac{1}{k!}\left(
\frac{-\tau}{2}\right)  ^{k}D^{2k}F(z).
\label{eq:heat_flow_operator_formal_series}%
\end{equation}
If $F(z)$ is a polynomial, the series terminates after finitely many nonzero
terms. In the simplest case when $F(z)=z^{n}$ and $\tau=1,$ this leads to the
\textit{Hermite polynomials} (with the probabilists' normalization) defined
by
\begin{equation}
\operatorname{He}_{n}(z):=\exp\left\{  -\frac{1}{2}D^{2}\right\}  z^{n}%
=n!\sum_{m=0}^{[n/2]}\frac{(-1)^{m}}{m!2^{m}}\cdot\frac{z^{n-2m}}%
{(n-2m)!},\qquad n\in\mathbb{N}_{0}. \label{eq:hermite_poly_def}%
\end{equation}

Section \ref{sec:heatOp} gives conditions under which the heat operator in
(\ref{eq:heat_flow_operator_formal_series}) for an entire function $F$ is well
defined and develops various properties of the solutions. Related results have been obtained by Papanicolaou, Kallitsi, and Smyrlis \cite{PKS}, with a focus on entire functions in the variable $\tau$ and hence excluding the heat evolved GAF (for which the heat flow is defined for $|\tau|<1$ only). Moreover, note that \eqref{eq:heat_flow_operator_formal_series} defines a solution to the (backward) heat equation, i.e.
$
\frac{\partial}{\partial \tau}F(z;\tau)=-\frac{1}{2}D^2F(z;\tau). %,\qquad t\geq0,\;x\in\mathbb{R}.
$
%
%If $\tau=-t,$
%for $t>0,$ the restriction of $F(z;t)$ to the real axis will solve the
%classical heat equation:
%\[
%\frac{\partial}{\partial t}F(x;-t)=\frac{1}{2}\cdot\frac{\partial^{2}%
%}{\partial z^{2}}F(x;-t),\qquad t\geq0,\;x\in\mathbb{R}.
%\]
%In the present paper, however, we allow the \textquotedblleft time\textquotedblright\ parameter $\tau$ to be complex. 
The heat flow operator
is closely related to the Segal--Bargmann
transform~\cite{hall_holomorphic_methods,folland_book} and the Weierstrass
transform, both of which have been much studied.

\subsection{Heat flow on the Gaussian analytic function}

Recently, there has been interest in studying the dynamics of the complex
zeros of polynomials (or entire functions) undergoing the heat flow;
see~\cite{csordas_smith_varga, tao_blog1,tao_blog2,rodgers_tao, kabluchko_curie_weiss,hallho,heatflowpoly}. The next theorem, which is one of our main results, shows that the GAF stays invariant
in distribution under a heat operator, followed by a dilation and
multiplication by a Gaussian (an exponential of a quadratic function). It follows that the distribution of zeros of the
GAF remains unchanged under the heat flow, up to a scaling. We will denote by $\mathcal{Z}(F)$ the collection of zeros of a function $F$ (which can also be identified with the point process $\sum_{z\in\mathcal Z (F)}\delta_{z}$) and we write $a\cdot\mathcal Z(F)=\{az :z\in\mathcal Z(F)\}$ for the element-wise rescaling by $a\in\mathbb C$.

\begin{theorem}
\label{theo:GAF_invariant_generalization_zeros_introduction} Let $G$
be the plane GAF and let $\tau$ be a complex number with $\left\vert
\tau\right\vert <1.$ Then the random holomorphic function $V_{\tau}G$ given by%
\begin{equation}
(V_{\tau}G)(z)=\left(  1-\left\vert \tau\right\vert ^{2}\right)
^{1/4}~e^{\bar\tau z^{2}/2}\left(  e^{-\tau D^{2}/2}G\right)  \left(
z\sqrt{1-\left\vert \tau\right\vert ^{2}}\right)  \label{eq:Vtau}%
\end{equation}
is well defined and has the same distribution as $G.$ In particular, we have equality in distribution of the collections of zeros: %that if
%$\{z_{j}\}=\mathcal Z (G)$ % denotes the collection of zeros of $G$
%and $\{z_{j}(\tau)\}=\mathcal Z(e^{-\tau D^{2}/2}G)$%denotes the collection of zeros of $e^{-\tau D^{2}/2}G$ , then the random collection of points
%\[
%\left\{  \frac{z_{j}(\tau)}{\sqrt{1-\left\vert \tau\right\vert ^{2}}%
%}\right\}\stackrel{d}{=}  \{z_{j}\}.
%\]
%or
\[
\frac{\mathcal Z(e^{-\tau D^{2}/2}G)}{\sqrt{1-\left\vert \tau\right\vert ^{2} }}=\mathcal Z (V_\tau G)\stackrel{d}{=}  \mathcal Z (G).
\]
%has the same distribution as the point process $\{z_{j}\}.$
\end{theorem}

Theorem \ref{theo:GAF_invariant_generalization_zeros_introduction} tells us
the distribution of the zeros of $e^{-\tau D^{2}/2}G$ for one fixed time
$\tau.$ It is, however, of considerable interest to understand how the zeros
of $e^{-\tau D^{2}/2}G$ evolve with time. The evolution of the zeros will
depend on \textit{where} in the complex plane the zeros are located. (Note that even though the zeros of the GAF form a stationary point process, the GAF itself is not stationary). For this
reason, it is convenient to condition the GAF to have a zero at a
fixed point $a\in\mathbb{C},$ and then track the evolution of this zero.

\begin{definition}
\label{zaTau.def}Fix $a\in\mathbb{C}$ and let $G^{a}$ be the GAF $G$
conditioned on $G(a)=0,$ so that $G^{a}$ has a zero at $a.$ With
probability 1, the zero of $G^{a}$ at $a$ is simple. Then define $z^{a}%
(\cdot)$ to be the unique holomorphic function, defined in a neighborhood of
0, such that $z^{a}(0)=a$ and such that
\[
(e^{-\tau D^{2}/2}G^{a})(z^{a}(\tau))=0.
\]
That is to say, $z^{a}(\tau)$ is the zero of $e^{-\tau D^{2}/2}G^{a}$ that
starts at $a$ when $\tau=0.$
\end{definition}

With this definition, we can state the following result.

\begin{theorem}
\label{theo:distributionZeros}We have the following equality in distribution:%
\begin{equation}\label{eq:distributionZeros}
z^{a}(\tau)\overset{d}{=}a+\tau\bar{a}+z^{0}(\tau).
\end{equation}

\end{theorem}

Here we emphasize that $z^{a}(\tau)$ is only defined for $\tau$ in
\textit{some} disk around $0$ and not necessarily for all $\tau$ in the unit
disk. After all, the zero $z^{a}(\tau)$ can collide with another zero and then
it will no longer possible to define $z^{a}(\tau)$ holomorphically. Thus, in
particular, $z^{0}(\tau)$ is a random variable defined on a disk whose radius is
also random. We conjecture that if $\tau$ is restricted to be real, then
$z^{a}(\tau)$ is well defined as a real analytic function of $\tau$ for all
$\tau\in(-1,1).$

Note that $z^{0}(\tau)$ is a fixed random variable whose distribution is
independent of $a.$ Thus, $z^{a}(\tau)$ is equal to $a+\tau\bar{a}$ plus an
error that is \textquotedblleft order 1\textquotedblright\ in the sense that
the error has a fixed probability distribution independent of $a.$ In
particular, the error will be small compared to $\left\vert a\right\vert $
when $\left\vert a\right\vert $ is large. We write this observation
schematically as%
\begin{equation}
z^{a}(\tau)\approx z^{a}(0)+\tau\overline{z^{a}(0)}. \label{eq:zjTraj}%
\end{equation}

Equation (\ref{eq:zjTraj}) says that if there happens to be a zero of the GAF
at a particular point $a$ in the plane, then that zero tends to move along a
straight line, with a fixed amount of variation, independent of $a.$ Figure
\ref{linearapprox.fig} shows this behavior in a simulation, which tracks the
motion of the 100 smallest zeros of the GAF under the heat flow for real time
$\tau$ with $0\leq \tau<1.$ The plot compares the actual trajectories\ (blue) to the
straight-line motion on the right-hand side of (\ref{eq:zjTraj}) (gray).
%\begin{remark}
%At the critical value $\tau=1$, one may guess from Figure \ref{linearapprox.fig} that all
%\end{remark}

%TCIMACRO{\FRAME{ftbpFU}{4.0266in}{2.1058in}{0pt}{\Qcb{Evolution of the 100
%smallest zeros of the GAF under the heat flow for real times $\tau$ with
%$0\leq\tau<1$ (blue). The straight-line approximation $a+\tau\bar{a}$ to each
%curve is shown in gray. The curves start in the disk of radius 10 and come
%close to the $x$-axis as $\tau$ approaches 1.}}{\Qlb{linearapprox.fig}%
%}{linearapprox.pdf}{\special{ language "Scientific Word";  type "GRAPHIC";
%width 4.0266in;  height 2.1058in;  depth 0pt;  original-width 6.2362in;
%original-height 12.0001in;  cropleft "0";  croptop "1";  cropright "1";
%cropbottom "0";  filename 'linearApprox.pdf';file-properties "XNPEU";}} }%
%BeginExpansion
\begin{figure}[ptb]%
\centering
\includegraphics[
%natheight=12.000100in,
%natwidth=6.236200in,
height=2.1058in,
width=4.0266in
]%
{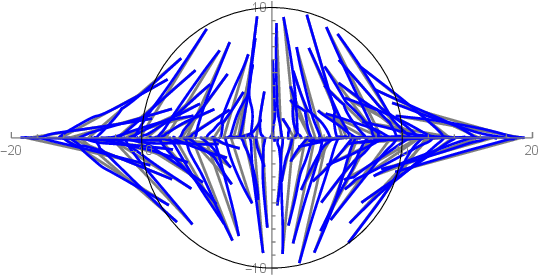}%
\caption{Evolution of the 100 smallest zeros of the GAF under the heat flow
for real times $\tau$ with $0\leq\tau<1$ (blue). The straight-line
approximation $a+\tau\bar{a}$ to each curve is shown in gray. The curves start
in the disk of radius 10 and come close to the $x$-axis as $\tau$ approaches
1.}%
\label{linearapprox.fig}%
\end{figure}
%EndExpansion

Note that we can rewrite \eqref{eq:distributionZeros} as%
\[
z^{a}(\tau)-\tau\overline{z^{a}(0)}\overset{d}{=}z^{a}(0)+z^{0}(\tau).
\]
That is to say, %if we take the term on the right-hand side of (\ref{eq:zjTraj}) that is linear in $\tau$ and move it to the other side,
the resulting curves $z^{a}(\tau)-\tau\overline{z^{a}(0)}$ should exhibit relatively small
motion around the initial point $a$; see Figure \ref{fluctuations.fig}.
Theorem \ref{theo:distributionZeros} says that the curves $z^{\alpha}%
(\tau)-\tau\overline{z^{a}(0)}$ all have the same distribution, namely that
of $z^{0}(\tau)$, except shifted by $z^{a}(0).$%

%TCIMACRO{\FRAME{ftbpFU}{3.5276in}{3.5181in}{0pt}{\Qcb{Plots of the curves
%$z_{j}(t)-\tau\overline{z_{j}(0)},$ where $z_{j}(0)$ is a zero of the GAF, for
%$0\leq\tau<1.$ There is a small dot at the starting point $z_{j}(0)$ of each
%curve. }}{\Qlb{fluctuations.fig}}{fluctuations.pdf}%
%width 3.5276in;  height 3.5181in;  depth 0pt;  original-width 5.0004in;
%original-height 4.9865in;  cropleft "0";  croptop "1";  cropright "1";
%cropbottom "0";  filename 'fluctuations.pdf';file-properties "XNPEU";}} }%
%BeginExpansion
\begin{figure}[ptb]%
\centering
\includegraphics[
%natheight=4.986500in,
%natwidth=5.000400in,
height=3.5181in,
width=3.5276in
]%
{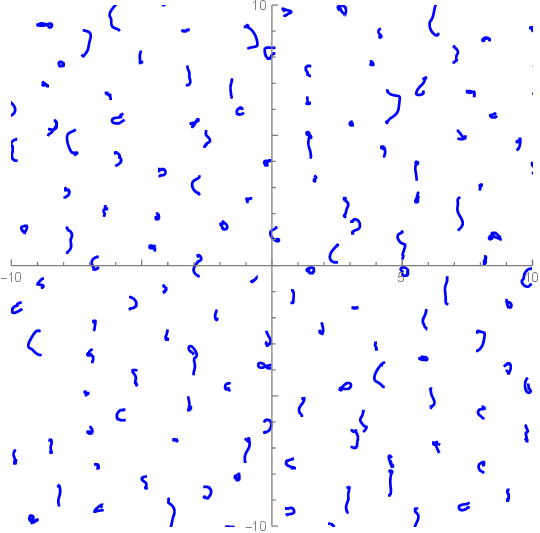}%
\caption{Plots of the curves $z_{j}(t)-\tau\overline{z_{j}(0)},$ where
$z_{j}(0)$ is a zero of the GAF, for $0\leq\tau<1.$ There is a small dot at
the starting point $z_{j}(0)$ of each curve. }%
\label{fluctuations.fig}%
\end{figure}
%EndExpansion

%We can also consider the zeros $\{w_{j}(\tau)\}$ of $V_{\tau}G,$ which are
%just rescalings of the zeros $\{z_{j}(\tau)\}$ of $e^{\tau D^{2}/2}G$:%
%\[
%\mathcal{Z}(V_{\tau}G)=\left\{  \left.  \frac{z_{j}(\tau)}{\sqrt{1-\left\vert
%\tau\right\vert ^{2}}}\right\vert z\in\mathcal{Z}(e^{\tau D^{2}/2}G)\right\}
%.
%\]
%%where $\mathcal{Z}(F)$ denotes the collection of zeros of a function $F.$
%Then we can define $w^{a}(\tau)$ by analogy to $z^{a}(\tau)$ (Definition \ref{zaTau.def})
%to be the zero of $V_{\tau}G^{a}$ that starts at $a$ at $\tau=0.$ Then as an
%immediate consequence of Theorem \ref{theo:distributionZeros}, we have the
%following result.
%
%\begin{corollary}
%\label{cor:distributionWtau}If $w^{a}(\tau)$ is the zero of $V_{\tau}G$ that
%starts at $a$ at $\tau=0,$ then we have the following equality in
%distribution:
%\[
%w^{a}(\tau)\overset{d}{=}\frac{a+\tau\bar{a}}{\sqrt{1-\left\vert
%\tau\right\vert ^{2}}}+w^{0}(\tau).
%\]
%
%\end{corollary}
%
%Since, again, $w^{0}(\tau)$ is a fixed random variable whose distribution is
%independent of $a,$ we may rewrite this result in rough form as%
%\begin{equation}
%w^{a}(\tau)\approx\frac{a+\tau\bar{a}}{\sqrt{1-\left\vert \tau\right\vert
%^{2}}}, \label{eq:waApprox}%
%\end{equation}
%where the approximation means that the error is small compared to $\left\vert
%a\right\vert $ when $\left\vert a\right\vert $ is large. Restricting $\tau$ to
%be a real number, (\ref{eq:waApprox}) says that $w^{a}(\tau)$ travels
%approximately along a hyperbola with $\operatorname{Re}(w)\operatorname{Im}%
%(w)$ equal to a constant.

\subsection{Connection to the group $SL(2;\mathbb{R})$\label{sec:sl2Intro}}

Recall that the Segal--Bargmann space over $\mathbb{C}$ is the Hilbert space
$\mathcal{B}$ of entire holomorphic functions $F$ such that
\begin{equation}
\left\Vert F\right\Vert ^{2}:=\int_{\mathbb{C}}\left\vert F(z)\right\vert
^{2}\frac{e^{-\left\vert z\right\vert ^{2}}}{\pi}~dz<\infty. \label{eq:SBdef}%
\end{equation}
The GAF is then the standard Gaussian measure based on the Segal--Bargmann
space $\mathcal{B}$, in the sense that the functions $z^{n}/\sqrt{n!}$
appearing in the definition (\ref{eq:GAFdef}) form an orthonormal basis for
$\mathcal{B}.$ Note, however, that the GAF does not actually live in
$\mathcal{B}$, since $\sum_{n}\left\vert \xi_{n}\right\vert ^{2}=\infty$
almost surely.

We let $SL(2;\mathbb{R})$ denote the group of real-linear transformations of
the plane with determinant 1. It is known that the zeros of the GAF are
invariant (in distribution) under rotations, but not\footnote{For example, the variance in the central limit theorem for linear statistics of the GAF zeros~\cite[p.~128]{sodin_tsirelson} is the $L^2$-norm of the Laplacian of the test function, which is not invariant under the transformation $(x,y) \mapsto (2x, y/2)$.} under general elements of
$SL(2;\mathbb{R}).$ On the other hand, there is a unitary representation of
the connected double cover of $SL(2;\mathbb{R})$ on the Segal--Bargmann space
known as the metaplectic representation. (See Appendix \ref{meta.appendix}
and, for example, Chapter 4 of \cite{folland_book}.) At a formal level, any
unitary $U$ transformation of $\mathcal{B}$ should induce a transformation
that leaves the GAF invariant in distribution. After all, $U$ just changes
the functions $z^{n}/\sqrt{n!}$ to another orthonormal basis for $\mathcal{B}$
and this should not affect the resulting Gaussian measure. In any case, we
will see directly that all operators in the metaplectic representation do
preserve the GAF\ in distribution, and that all the metaplectic operators can
be built up from the $V_{\tau}$'s and rotations.

The metaplectic representation associates to each $A\in SL(2;\mathbb{R})$ a
pair of operators $\pm V(A)$ differing by a sign and satisfying%
\[
V(AB)=\pm V(A)V(B),\quad A,B\in SL(2;\mathbb{R}).
\]
Let $\mathbb{D}=\{\tau\in\mathbb{C}:|\tau|<1\}$ denote the unit disk.
\begin{proposition}
\label{prop:metaIntro} %We have the following results about the operators $V(A).$

\begin{enumerate}
\item Suppose $\tau\in\mathbb{D}$ and consider the real-linear
transformation of $\mathbb{C}$ to $\mathbb{C}$ given by%
\begin{equation}
z\mapsto\frac{z}{\sqrt{1-\left\vert \tau\right\vert ^{2}}}+\frac{\tau}%
{\sqrt{1-\left\vert \tau\right\vert ^{2}}}\bar{z}, \label{AtauMap}%
\end{equation}
which is represented by the matrix $A_{\tau}$ in $SL(2;\mathbb{R})$ given by%
\begin{equation}
A_{\tau}=\frac{1}{\sqrt{1-\left\vert \tau\right\vert ^{2}}}\left(
\begin{array}
[c]{cc}%
1+\operatorname{Re}\tau & \operatorname{Im}\tau\\
\operatorname{Im}\tau & 1-\operatorname{Re}\tau
\end{array}
\right)  . \label{Atau}%
\end{equation}
Then the sign of $V(A)$ can be chosen so that%
\[
V(A_{\tau})=V_{\tau},
\]
where $V_{\tau}$ is as in (\ref{eq:Vtau}).

\item If
\[
A=\left(
\begin{array}
[c]{rr}%
\cos\theta & -\sin\theta\\
\sin\theta & \cos\theta
\end{array}
\right)
\]
is a rotation by angle $\theta,$ then $V(A)$ acts by rotations, up to a phase
factor. Specifically,
\[
(V(A)F)(z)=\pm e^{-i\theta/2}F(e^{-i\theta}z).
\]

\end{enumerate}
\end{proposition}

The matrix $A_{\tau}$ is symmetric and positive definite for all $\tau$ in the
unit disk $\mathbb{D},$ and as $\tau$ varies over $\mathbb{D},$ every
symmetric, positive element of $SL(2;\mathbb{R})$ arises. Since every element
of $SL(2;\mathbb{R})$ can be written as the product of a rotation and a
positive, symmetric element, the proposition tells us how to compute
\textit{all} the elements $V(A),$ $A\in SL(2;\mathbb{R}).$

\begin{proposition}
We have the following results.

\begin{enumerate}
\item For every $A\in SL(2;\mathbb{R}),$ the operator $V(A)$ preserves the GAF
in distribution.

\item Suppose we factor $A$ uniquely as%
\begin{equation}
A=\left(
\begin{array}
[c]{rr}%
\cos\theta & -\sin\theta\\
\sin\theta & \cos\theta
\end{array}
\right)  A_{\tau}, \label{Afact}%
\end{equation}
where $A_{\tau}$ is a positive symmetric matrix written in the form
(\ref{Atau}), and we write the zeros of $e^{-\tau D^{2}/2}G$ as $\{z_{j}%
(\tau)\}.$ Then the zeros $\mathcal{Z}(V(A)G)$ of $V(A)G$ are related to the
zeros of $e^{-\tau D^{2}/2}G$ by
\[
\mathcal{Z}(V(A)G)=\left\{  \left.  e^{i\theta}\frac{z}{\sqrt{1-\left\vert
\tau\right\vert ^{2}}}\right\vert z\in\mathcal{Z}(e^{-\tau D^{2}/2}G)\right\}
.
\]

\item Recall the definition of $z^{a}(\tau)$ in Definition \ref{zaTau.def} and
define $u^{a}(A)$ by%
\[
u^{a}(A)=e^{i\theta}\frac{z^{a}(\tau)}{\sqrt{1-\left\vert \tau\right\vert
^{2}}},
\]
where $A\in SL(2;\mathbb{R)}$ is factored as in (\ref{Afact}). Then we have
the following equality in distribution%
\begin{equation}
u^{a}(A)=Aa+u^{0}(A), \label{uaDist}%
\end{equation}
where $Aa$ denotes the action of $A$ on $a,$ where $a$ is viewed as an element
of $\mathbb{R}^{2}.$
\end{enumerate}
\end{proposition}

As in (\ref{eq:zjTraj}) %and (\ref{eq:waApprox}),
we rewrite (\ref{uaDist}) in
rough form as%
\[
u^{a}(A)\approx Aa,
\]
to indicate that the error term $u^{0}(A)$ has a fixed distribution whose size
is independent of $A.$ Thus, in a precisely formulated sense (given in
(\ref{uaDist})), the action of $V(A)$ on the zeros of the GAF is \textit{approximately} 
the ordinary action of $SL(2;\mathbb{R})$ on $\mathbb{C\cong R}^{2}.$

\section{Heat flow on entire functions of order at most 2\label{sec:heatOp}}

In this section, we define the heat flow on entire functions
(with suitable growth rates at infinity) and state some of its properties.
Related results have been obtained by Papanicolaou, Kallitsi, and Smyrlis in
\cite{PKS}. The main difference between our results and theirs is that they
focus on solutions $F(z,\tau)$ of the heat equation that are entire
holomorphic functions of \textit{both} $z$ and $\tau.$ By contrast, we
consider solutions that are entire in $z$ but may be defined only for $\tau$
in a disk of finite radius. This approach allows us to consider functions
where $F(z,0)$ has order 2 and arbitrary finite type, whereas in \cite{PKS}
the functions are assumed to be of order less than 2 or order 2 and type 0.
Our approach is motivated by applications to the Gaussian analytic function,
which has order 2 and type $1/2.$ \newline

The reader who is interested primarily in the results about the Gaussian
analytic function may skim most of this section, but noting one key
result (Theorem \ref{theo:entire_function_heat_flow_exists}): that the heat
operator $e^{-\tau D^{2}/2}$ makes sense when applied to an entire function of
order 2 and type $\sigma,$ provided that $\left\vert \tau\right\vert
<1/(2\sigma).$

\subsection{Order and type of a holomorphic function}

Let us first recall some notions from the theory of entire functions;
see~\cite{levin_book_lectures,levin_book_distribution,boas_book},
\cite[Chapter~XI]{conway_book1} for more details. Let $F(z)=\sum_{n=0}%
^{\infty}a_{n}z^{n}$ be an entire function. Define 
\[M(r)=\sup_{|z|=r}|F(z)|,\quad r>0.
\] 
The \emph{order} of $F$, denoted by $\rho\in\lbrack0,+\infty]$, is the
infimum of all numbers $\mu>0$ with the property that $M(r)<\exp\{r^{\mu}\}$
for all sufficiently large $r>0$. If the order $\rho$ is finite, then the
\emph{type} of $F$, denoted by $\sigma\in\lbrack0,+\infty]$, is the infimum of
all numbers $a\geq0$ with the property that $M(r)\leq\exp\{ar^{\rho}\}$ for
all sufficiently large $r>0$. Equivalently, we have
\[
\rho=\limsup_{r\rightarrow\infty}\frac{\log\log M(r)}{\log r}\in
\lbrack0,+\infty],\qquad\sigma=\limsup_{r\rightarrow\infty}\frac{\log
M(r)}{r^{\rho}}\in\lbrack0,+\infty].
\]
It is known, see~\cite[Theorems~2,3 on p.~6]{levin_book_lectures}
or~\cite[Theorem~2 on p.~4]{levin_book_distribution}, that
\begin{equation}
\rho=\limsup_{n\rightarrow\infty}\frac{n\log n}{\log|1/a_{n}|},\qquad
(\sigma\mathrm{e}\rho)^{1/\rho}=\limsup_{n\rightarrow\infty}n^{1/\rho}%
|a_{n}|^{1/n}. \label{eq:lindeloef_pringsheim}%
\end{equation}
where for the second formula we assume that the order $\rho$ is finite and non-zero.

Using these formulae one easily checks that for the plane GAF we have $\rho=2$
and $\sigma=1/2$ a.s. More precisely, we recall that $a_n = \xi_n / \sqrt{n!}$, note that $\log (n!) = n \log n - n + o(n)$ by the Stirling formula and $\log |\xi_n| = o(n)$ a.s.,\ which follows from the Borel-Cantelli lemma together with the observation that the random variables $|\xi_0|^2, |\xi_1|^2,\ldots$ are standard exponential.

\subsection{Definition of the heat flow}

The main result of this section is that we can make sense of the heat operator applied to a holomorphic function $F$ in certain cases. %(1) if $F$ has order
%$\rho<2$, or (2) $F$ has order $\rho=2$ and finite type $\sigma,$ where in the
%second case, we require that
%\begin{equation}
%\left\vert \tau\right\vert <\frac{1}{2\sigma}. \label{tauBound}%
%\end{equation}

%We continue to use the notation
%\[
%D=\frac{d}{dz}%
%\]
%for the differentiation operator on holomorphic functions. We wish to make
%sense of the heat operator%
%\begin{equation}
%\exp\left\{  -\frac{\tau}{2}D^{2}\right\}  \label{HeatOp}%
%\end{equation}
%applied to a holomorphic function with appropriate growth bounds at infinity.
%We emphasize that this is the heat operator in one complex variable, which should not be confused with the heat operator in two real variables, which can be expressed as $\exp\{\frac{t}{2}\frac{\partial^{2}}{\partial z\partial \bar{z}}\}$. (This last operator, when applied to a holomorphic function $F$, gives simply $F$, because $\partial F/\partial\bar{z}=0.$)

%If we apply the heat operator (\ref{HeatOp}) to a (holomorphic) polynomial, we
%can compute the result as a terminating power series in powers of $D^{2},$ as
%in the case of the Hermite polynomials in (\ref{eq:hermite_poly_def}).

\begin{notation}
\label{s1s2.notation}For an entire holomorphic function $F,$ we will
distinguish two situations:

\emph{(S1)}: $F$ is an entire function of order $\rho<2$

\emph{(S2)}: $F$ is an entire function of order $\rho=2$ and finite type $\sigma
\geq0$.
\end{notation}

We keep in mind that in both situations there is some finite $\sigma_{0}\geq0$
such that for every $\varepsilon>0$ we have a growth of order
\begin{equation}
M(r)\leq\mathrm{e}^{(\sigma_{0}+\varepsilon)r^{2}} \label{eq:assumption_F}%
\end{equation}
provided $r>r(\varepsilon)$ is sufficiently large. In situation (S1),
\eqref{eq:assumption_F} holds with $\sigma_{0}=0$, while for (S2) it holds
with $\sigma_{0}=\sigma$.
%It follows from Theorem~\ref{theo:entire_function_heat_flow_exists} that, under~\eqref{eq:assumption_F}, $(z,t)\mapsto \exp\{-\frac t2 D^2\} F (z)$ defines an analytic function of two variables in the domain $\{(z,t)\in  \mathbb{C}^2: z\in \mathbb{C}, |t|<1/(2\sigma_0)\}$ (where $1/0 = +\infty$).

In this section, we will show that the heat operator can be
defined on any holomorphic function $F$ in situation (S1) and on a holomorphic
function $F$ in situation (S2) provided that $\left\vert \tau\right\vert
<1/(2\sigma).$ We will give four equivalent ways of the defining $\exp\{-\tau
D^{2}/2\}F$ in these cases:

\begin{itemize}
\item As a power series in powers of $D^{2},$

\item As a term-by-term action on the Taylor series of $F,$

\item As an integral operator over the real line, and

\item As an integral operator over the plane with respect to a Gaussian measure.
\end{itemize}

%We write $z=x+iy$ and denote the Wirtinger derivative as
%$$
%D=\frac{\partial}{\partial z}=\frac 12\left(\frac{\partial}{\partial_x}-i\frac{\partial}{\partial_y}\right).
%$$
The two most obvious ways to define $(e^{-\tau D^{2}/2}F)(z)$ for an
entire function $F(z)=\sum_{n=0}^{\infty}a_{n}z^{n}$ are the following. First, to
develop $e^{-\tau D^{2}/2}$ into a formal Taylor series in powers of $D^{2}$
and apply each term to $F$. Second, to apply the operator $e^{-\tau
D^{2}/2}$ termwise to each summand $a_{n}z^{n}$ using the formula
\begin{equation}
e^{-\tau D^{2}/2}(z^{n})=\tau^{n/2}\operatorname{He}_{n}\left(  \frac{z}%
{\sqrt{\tau}}\right)  ,\qquad\tau\in\mathbb{C},
\label{eq:eq:hermite_exp_on_z^n_with_sigma}%
\end{equation}
where $\operatorname{He}_{n}$ denotes the $n$-th Hermite polynomial defined
by~\eqref{eq:hermite_poly_def} and the right-hand side
of~\eqref{eq:eq:hermite_exp_on_z^n_with_sigma} is a polynomial of $\tau$ (not
of $\sqrt{\tau}$), as follows from~\eqref{eq:hermite_poly_def}. Here and in
the following we view the right hand side for $\tau=0$ as the analytic
continuation of $\tau^{n/2}\operatorname{He}_{n}\left(  \frac{z}{\sqrt{\tau}%
}\right)  $, namely $z^{n}$. Both ways to define the heat flow of $F$ lead to
the same result, as the next theorem shows.

\begin{theorem}
\label{theo:entire_function_heat_flow_exists} Let $F(z)=\sum_{n=0}^{\infty
}a_{n}z^{n}$ be an entire function satisfying (S1) or (S2). Let $\tau$ be a
complex number and assume $\left\vert \tau\right\vert <1/(2\sigma)$ in
situation (S2). Then the two series
\begin{align}
e^{-\tau D^{2}/2}F(z)  &  =\sum_{k=0}^{\infty}\frac{1}{k!}\left(  -\frac{\tau
}{2}D^{2}\right)  ^{k}F(z),\label{eq:heat_flow_def_1}\\
e^{-\tau D^{2}/2}F(z)  &  =\sum_{n=0}^{\infty}a_{n}(\sqrt{\tau})^{n}%
\operatorname{He}_{n}\left(  \frac{z}{\sqrt{\tau}}\right)  .
\label{eq:heat_flow_def_2}%
\end{align}
both converge to the same limit for all $z\in\mathbb{C}.$ In
(\ref{eq:heat_flow_def_2}), either of the two square roots of $\tau$ may be
used. The convergence is absolute and uniform in $z$ and $\tau$ if

(i) $\left\vert z\right\vert \leq C_{1}$ and $\left\vert \tau\right\vert \leq
C_{2}$ for any constants $C_{1},C_{2}>0$ in situation (S1),

(ii) $\left\vert z\right\vert \leq C_{1}$ and $\left\vert \tau\right\vert \leq
C_{2}$ for any $C_{1}>0$ and any $C_{2}<1/(2\sigma)$ in situation (S2).

In particular, both series define the same entire function of the two
variables $(z,\tau)$ in the domain $\mathbb{C}^{2}$ for (S1) and $\mathbb{C}%
\times\{|\tau|<1/(2\sigma)\}$ for (S2).
\end{theorem}

\begin{proof}
We start by analyzing~\eqref{eq:heat_flow_def_2}. To this end, we need the
well-known asymptotic equivalence
\[
\operatorname{He}_{n}(x)\sim e^{x^{2}/4}2^{n/2}\pi^{-1/2}\Gamma\left(
\frac{n+1}{2}\right)  \cos\left(  x\sqrt{n}-\frac{\pi n}{2}\right)  ,
\]
which holds as $n\rightarrow\infty$ locally uniformly in $x\in\mathbb{C}$, see
\cite[\S 18.15(v)]{NIST} and~\cite[Theorem~8.22.7]{szegoe_book}. Observing
that $\cos(\ldots)=\mathrm{e}^{o(n)}$ and using the Stirling formula, we can
write
\[
\operatorname{He}_{n}(x)=(n/\mathrm{e})^{n/2}\mathrm{e}^{o(n)},\qquad
n\rightarrow\infty,
\]
locally uniformly in $x\in\mathbb{C}$. Assume that $\rho<2$. Then,
by~\eqref{eq:lindeloef_pringsheim} there is an $\varepsilon>0$ such that
$|a_{n}|\leq n^{-(1+\varepsilon)n/2}$ for all sufficiently large $n$. It
follows that
\[
\left\vert a_{n}\tau^{n/2}\operatorname{He}_{n}\left(  \frac{z}{\sqrt{\tau}%
}\right)  \right\vert \leq e^{o(n)}|\tau|^{n/2}(n/e)^{n/2}n^{-(1+\varepsilon
)n/2}=n^{-\varepsilon n/2}e^{O(n)},
\]
which is a summable sequence. The $O$-term is uniform in $|z|<C$ and $|\tau|<C$
as long as $\tau$ stays away from $0$. The uniform absolute convergence of the series~\eqref{eq:heat_flow_def_2} when $\tau$ is in a small disk
around $0$ can be proved by obtaining an estimate using the Cauchy formula. This proves part~(i) of the
theorem for the series~\eqref{eq:heat_flow_def_2}.

Assume now that $\rho=2$ and, additionally, the type $\sigma$ is finite. Then,
by~\eqref{eq:lindeloef_pringsheim} we have
\[
|a_{n}|\leq n^{-n/2}(2e\sigma+o(1))^{n/2}.
\]
It follows that
\begin{align*}
\left\vert a_{n}\tau^{n/2}\operatorname{He}_{n}\left(  \frac{z}{\sqrt{\tau}%
}\right)  \right\vert &\leq n^{-n/2}(2e\sigma+o(1))^{n/2}\cdot|\tau|^{n/2}%
\cdot(n/e)^{n/2}e^{o(n)}\\
&=(2\sigma+o(1))^{n/2}|\tau|^{n/2},
\end{align*}
which is summable for $|\tau|<1/(2\sigma)$. For $\sigma\in(0,\infty)$,the
convergence of the series $\sum (2\sigma+o(1))^{n/2}|\tau|^{n/2}$ is uniform as long as $\varepsilon<|\tau|<1/(2\sigma)-\varepsilon$
and $|z|<C$ for some $\varepsilon>0$ and $C>0$. For $\sigma=0$, the convergence is uniform
as long as $\varepsilon<|\tau|<C$ and $|z|<C$ for some $\varepsilon>0$ and
$C>0$. This shows that the convergence of the right-hand side of~\eqref{eq:heat_flow_def_2} is absolute and uniform as long as $\varepsilon<|\tau|<1/(2\sigma)-\varepsilon$
and $|z|<C$ in the $\sigma\in(0,\infty)$ case and as long as $\varepsilon<|\tau|<C$ and $|z|<C$ for some $\varepsilon>0$ and $C>0$ in the $\sigma = 0$ case. To complete the uniform absolute convergence of the right-hand side of~\eqref{eq:heat_flow_def_2}, we can obtain an estimate to the domain $|\tau|\leq\varepsilon$ using Cauchy formula. This proves part~(ii) of the theorem for the series~\eqref{eq:heat_flow_def_2}.

To argue that the right-hand sides of~\eqref{eq:heat_flow_def_1}
and~\eqref{eq:heat_flow_def_2} are equal and to prove the theorem for the
series~\eqref{eq:heat_flow_def_2}, we consider the formal double series
\begin{equation}
\mathrm{e}^{-\tau D^{2}/2}F(z)=\sum_{n=0}^{\infty}\sum_{m=0}^{[n/2]}a_{n}%
\frac{(-1)^{m}n!}{2^{m}m!(n-2m)!}z^{n-2m}\tau^{m}. \label{eq:heat_flow_def_3}%
\end{equation}
and observe that both, \eqref{eq:heat_flow_def_1}
and~\eqref{eq:heat_flow_def_2}, can be obtained by re-grouping the terms
in~\eqref{eq:heat_flow_def_3}. It suffices to show
that~\eqref{eq:heat_flow_def_3} converges absolutely and locally uniformly on
the respective domain of the variables $(z,t)$. Taking the absolute values of
the terms in~\eqref{eq:heat_flow_def_3} results in the series
\[
\sum_{n=0}^{\infty}\sum_{m=0}^{[n/2]}|a_{n}|\frac{n!}{2^{m}m!(n-2m)!}%
|z|^{n-2m}|\tau|^{m}=\sum_{n=0}^{\infty}|a_{n}||\tau|^{n/2}\operatorname{He}%
_{n}\left(  \frac{|z|}{\sqrt{|\tau|}}\right)  ,
\]
where the right-hand side is obtained by re-grouping the terms in the
left-hand side, which is justified because all terms are non-negative. The
absolute and locally uniform (in the respective domain) convergence of this
series has been shown above.
\end{proof}

When $\tau$ is real and positive, can restrict $F$ to the real axis, apply the
usual real-variables heat operator (computed as convolution with a Gaussian),
and then holomorphically extend the result back to the complex plane.

\begin{theorem}
\label{prop:heat_flow_as_convolution}Suppose $\tau$ is real and positive,
where we assume $\tau<1/(2\sigma)$ in situation (S2). Then (reversing the
usual sign in the exponent of the heat operator) we have
\begin{equation}
(e^{+\tau D^{2}/2}F)(z)=\frac{1}{\sqrt{2\pi\tau}}\int_{\mathbb{R}%
}e^{-(z-x)^{2}/(2\tau)}F(x)~dx,\quad\tau>0,~z\in\mathbb{C}.
\label{eq:convolveR}%
\end{equation}
For general $\tau=\left\vert \tau\right\vert e^{i\theta}$ (satisfying
$\left\vert \tau\right\vert <1/(2\sigma)$ in situation (S2)), we have
(reverting to our usual sign in the exponent)%
\begin{equation}
(e^{-\tau D^{2}/2}F)(z)=\frac{1}{\sqrt{2\pi\left\vert \tau\right\vert }}%
\int_{\mathbb{R}}\exp\left\{  -\left(  -i e^{-i\theta/2}z-x\right)
^{2}/(2\left\vert \tau\right\vert )\right\}  F(ie^{i\theta/2}x)~dx.
\label{eq:convolveR2}%
\end{equation}

\end{theorem}

One could also attempt to analytically continue the formula
(\ref{eq:convolveR}) in $\tau$ %that is, to replace $t$ by $\tau$
using the principal branch of the square root. This approach is valid under suitable
assumptions to ensure convergence of the integral. We will not, however,
obtain convergence for all $\tau$ with $\left\vert \tau\right\vert
<1/(2\sigma),$ even in situation (S1), for example if $\operatorname{Re}\tau$
is negative and $F$ grows along the real axis.

\begin{proof}
We start by verifying (\ref{eq:convolveR}). Observe that for all $\tau>0,$ we
have
\begin{align*}
\frac{1}{\sqrt{2\pi\tau}}\int_{-\infty}^{\infty}(z-x)^{n}\mathrm{e}%
^{-x^{2}/(2\tau)}d\,x  &  =\frac{1}{\sqrt{2\pi\tau}}\sum_{\ell=0}^{\lfloor
n/2\rfloor}\binom{n}{2\ell}z^{n-2\ell}\int_{-\infty}^{+\infty}x^{2\ell
}\mathrm{e}^{-x^{2}/(2\tau)}\,dx\\
&  =\sum_{\ell=0}^{\lfloor n/2\rfloor}\binom{n}{2\ell}z^{n-2\ell}\tau^{\ell
}\frac{(2\ell)!}{2^{\ell}\ell!}=(-\tau)^{n/2}\operatorname{He}_{n}\left(
\frac{z}{\sqrt{-\tau}}\right)  .
\end{align*}
Writing $F(z)=\sum_{n=0}^{\infty}a_{n}z^{n}$ and assuming that the integral
and the sum can be interchanged, we have
\begin{align*}
\frac{1}{\sqrt{2\pi\tau}}\int_{-\infty}^{\infty}F(z-x)\mathrm{e}%
^{-x^{2}/(2\tau)}\,dx  &  =\sum_{n=0}^{\infty}\frac{a_{n}}{\sqrt{2\pi\tau}%
}\int_{\mathbb{R}}(z-x)^{n}\mathrm{e}^{-x^{2}/(2\tau)}\,dx\\
&  =\sum_{n=0}^{\infty}a_{n}(-\tau)^{n/2}\operatorname{He}_{n}\left(  \frac
{z}{\sqrt{-\tau}}\right)  =\mathrm{e}^{+\frac{\tau}{2}D^{2}}F(z),
\end{align*}
as claimed.

To justify the interchanging of the integral and the sum, it suffices to check
that
\[
S:=\sum_{n=1}^{\infty}\int_{-\infty}^{\infty}|a_{n}||z-x|^{n}\mathrm{e}%
^{-\frac{x^{2}}{2}\operatorname{Re}(1/\tau)}\,dx<\infty.
\]
Let $C>0$ be a sufficiently large constant whose value may change from line to
line. Then, $|z-x|^{n}\leq C+C|x|^{n}$ for all $x\in\mathbb{R}$. First,
consider the case when $\rho=2$ and $\operatorname{Re}(1/\tau)>2\sigma$ if
$\sigma>0$ (the case $\rho<2$ being simpler). For every $\varepsilon>0$ we
have $|a_{n}|\leq Cn^{-n/2}(2\sigma\mathrm{e}+\varepsilon)^{n/2}$ for all
$n\in\mathbb{N}$ by~\eqref{eq:lindeloef_pringsheim}, and it follows that
\begin{align*}
S&\leq\sum_{n=1}^{\infty}\left(  C|a_{n}|+C|a_{n}|\int_{-\infty}^{\infty
}|x|^{n}\mathrm{e}^{-\frac{x^{2}}{2}\mathbb{R}e(1/\tau)}\,dx\right)  \\
&\leq
C+C\sum_{n=1}^{\infty}\frac{(2\sigma\mathrm{e}+\varepsilon)^{n/2}}{n^{n/2}%
}\frac{2^{n/2}\Gamma\left(  \frac{n+1}{2}\right)  }{(\operatorname{Re}%
(1/\tau))^{n/2}},
\end{align*}
which is finite (for sufficiently small $\varepsilon>0$) by an application of
the Stirling formula.

The general result in (\ref{eq:convolveR2}) follows easily using Point 3 of
Lemma \ref{lem:heat_flow_properties1}, once we establish Lemma \ref{lem:heat_flow_properties1} later in the paper.
\end{proof}

Our last formula for the heat operator is based on the Segal--Bargmann space,
defined as follows.

\begin{definition}
\label{def:SBspace}The Segal--Bargmann space over $\mathbb{C}$ is the space of
entire holomorphic functions $F$ satisfying%
\[
\int_{\mathbb{C}}\left\vert F(z)\right\vert ^{2}\frac{e^{-\left\vert
z\right\vert ^{2}}}{\pi}~dz<\infty,
\]
where $dz$ is the two-dimensional Lebesgue measure on $\mathbb{C}%
\cong\mathbb{R}^{2}.$ This space is a closed subspace of the Hilbert space
$L^{2}(\mathbb{C},e^{-\left\vert z\right\vert ^{2}}/\pi~dz).$
\end{definition}

Elements of the Segal--Bargmann space are of order $\rho\leq2$ with type at
most $1/2$ in the case $\rho=2.$ Consideration of the Segal--Bargmann space is
natural because the law of the GAF\ is the Gaussian measure based on the
Segal--Bargmann space. Concretely, this statement means that the functions
$z^{n}/\sqrt{n!}$ appearing in the definition (\ref{eq:GAFdef}) of the GAF
form an orthonormal basis for the Segal--Bargmann space. We emphasize,
however, that the GAF does not actually belong to the Segal--Bargmann space,
because in (\ref{eq:GAFdef}), $\sum_{n}\left\vert \xi_{n}\right\vert
^{2}=\infty$ with probability 1.

We may apply the heat operator to elements of the Segal--Bargmann space
provided that $\left\vert \tau\right\vert <1.$ Indeed, the map%
\[
F\mapsto(e^{-\tau D^{2}/2}F)(z)
\]
is a continuous linear functional on the Segal--Bargmann space for each fixed
$z\in\mathbb{C}$ and $\tau$ with $\left\vert \tau\right\vert <1.$ This linear
functional may therefore be written as an inner product with a unique element
of the Segal--Bargmann space. This line of reasoning gives another integral
representation of the heat operator---and this representation is actually
valid for functions satisfying (S1) or satisfying (S2) with $\sigma\leq1/2.$

\begin{theorem}
\label{theo:heatIntOp}Suppose $F$ satisfies (S1) or (S2), where in the case
(S2) we assume that the type $\sigma$ is at most $1/2.$ Then for all $\tau$
with $\left\vert \tau\right\vert <1,$ we have%
\begin{equation}
\left(  e^{-\tau D^{2}/2}F\right)  (z)=\int_{\mathbb{C}}\exp\left\{
-\frac{\tau}{2}\bar{w}^{2}+z\bar{w}\right\}  F(w)\frac{e^{-\left\vert
w\right\vert ^{2}}}{\pi}~dw. \label{SBheat}%
\end{equation}

\end{theorem}

When $\tau=0,$ the formula (\ref{SBheat}) reduces to the reproducing kernel
formula for $F(z)$:%
\begin{equation}
F(z)=\int_{\mathbb{C}}\exp\left\{  z\bar{w}\right\}  F(w)\frac{e^{-\left\vert
w\right\vert ^{2}}}{\pi}~dw, \label{reproKernel}%
\end{equation}
as in \cite[Theorem 3.4]{hall_holomorphic_methods}.

\begin{proof}
We start with the case where $\tau$ is real and negative and write%
\begin{equation}
(e^{-\tau D^{2}/2}F)(z)=\frac{1}{\sqrt{-2\pi\tau}}\int_{\mathbb{R}%
}e^{(z-x)^{2}/(2\tau)}F(x)~dx. \label{SBheatStep1}%
\end{equation}
We then use the reproducing kernel formula (\ref{reproKernel}):%
\[
F(x)=\int_{\mathbb{C}}e^{x\bar{w}}F(w)\frac{e^{-\left\vert w\right\vert ^{2}}%
}{\pi}~dw.
\]
Inserting this expression into (\ref{SBheatStep1}) and reversing the order of
integration gives%
\begin{equation}
(e^{-\tau D^{2}/2}F)(z)=\int_{\mathbb{C}}\left(  \frac{1}{\sqrt{2\pi(-\tau)}%
}\int_{\mathbb{R}}e^{(z-x)^{2}/(2\tau)}e^{x\bar{w}}~dx\right)  F(w)\frac
{e^{-\left\vert w\right\vert ^{2}}}{\pi}~dw. \label{SBheatStep2}%
\end{equation}
(We leave it as an exercise to verify that Fubini's theorem applies.) The
inner integral in (\ref{SBheatStep2}) is the heat operator $e^{-\tau D^{2}/2}$
applied to $e^{z\bar{w}},$ which gives $e^{-\frac{\tau}{2}\bar{w}^{2}}%
e^{z\bar{w}},$ since $De^{z\bar{w}}=\bar{w}e^{z\bar{w}}.$

We therefore obtain the desired result when $\tau$ is real and negative. We
then observe that both sides of the desired equality (\ref{SBheat}) are
well-defined holomorphic functions of $\tau$ for $\left\vert \tau\right\vert
<1.$ Since they agree when $\tau$ is real and negative, they agree for all
$\tau.$
\end{proof}

\subsection{Properties of the heat flow}

In a first lemma we collect some basic properties of the heat flow operator.

\begin{lemma}
\label{lem:heat_flow_properties1} Let $F$ be an entire function satisfying
(S1) or (S2) and $\tau\in\mathbb{C}$ such that in the case (S2), we have
$\left\vert \tau\right\vert <1/(2\sigma)$. Then, the following hold:

\begin{enumerate}
\item $e^{-\tau D^{2}/2}$ commutes with $D$, that is $\mathrm{e}^{-\tau
D^{2}/2}DF=De^{-\tau D^{2}/2}F$.

\item $e^{-\tau D^{2}/2}$ commutes with translations. More precisely, for
$a\in\mathbb{C}$ we have
\[
e^{-\tau D^{2}/2}[F(a+\cdot)](z)=e^{-\tau D^{2}/2}[F(\cdot)](a+z).
\]
%we write $(T_a F) (z) = F(z-a)$, then $\mathrm{e}^{-\frac t2 D^2} T_a F = T_a \mathrm{e}^{-\frac t2 D^2} F$.

\item For every $\lambda\in\mathbb{C}$ such that in the case (S2) it holds
$\left\vert \lambda^{2}t\right\vert <1/(2\sigma)$ we have
\[
\mathrm{e}^{-\tau D^{2}/2}[F(\lambda\cdot)](z)=\mathrm{e}^{-\tau\lambda
^{2}D^{2}/2}[F(\cdot)](\lambda z).
\]

\item The function $F(z;t):=(e^{-\tau D^{2}/2}F)(z)$ solves the (backward)
heat equation
\[
\frac{\partial}{\partial t}F(z;t)=-\frac{1}{2}\cdot\left(  \frac{d}%
{dz}\right)  ^{2}F(z;t).
%\qquad s\in \mathbb{C}, \; |s| < 1/ 2\sigma_0), \; z\in \mathbb{C}.
\]

\end{enumerate}
\end{lemma}

Here, we write $\exp\{-\tau D^{2}/2\}[F(\cdot)](\varphi(z))$ if we want to
apply the heat-flow operator to $F(z)$ and then to replace $z$ by $\varphi(z)$
in the resulting function. If we want to apply the heat-flow operator to
$F(\varphi(z))$, we write $\exp\{-\tau D^{2}/2\}[F(\varphi(\cdot))](z)$.

%\JJ{I have split and rearranged the properties from the previous version. Here are the easier properties which follow immediately. Also, we need Property (c) to "spin the wheel" in the next Proposition (about convolution), which in turn is needed to prove the last properties below about order and type.}

\begin{proof}
These basic properties follow immediately from the absolute convergence of the
series \eqref{eq:heat_flow_def_1} in the respective domain. We leave the
details as an exercise to the reader.
\end{proof}

Our next insight reveals that the situations (S1) and (S2) stay invariant
under the heat flow.

\begin{proposition}
\label{prop:heat_flow_properties2} Let $F$ be an entire function satisfying
(S1) or (S2).

\begin{enumerate}
\item If $F$ is of order $\rho<2$ then $e^{-\tau D^{2}/2}F$ is of order
$\rho<2$ for all $\tau\in\mathbb{C}$.

\item If $F$ is of order $\rho=2$ and finite type $\sigma$, then for
$\left\vert \tau\right\vert <1/(2\sigma)$ the entire function $e^{-\tau
D^{2}/2}F$ is of order $\rho=2$ and of type at most $\sigma(\tau
)=\sigma/(1-2\sigma\left\vert \tau\right\vert )$.

\item Let $\tau_{1},\tau_{2}\in\mathbb{C}$ be such that for (S2) we assume
$\left\vert \tau_{1}\right\vert +\left\vert \tau_{2}\right\vert <1/(2\sigma)$.
Then, $e^{-\tau_{1}D^{2}/2}e^{-\tau_{2}D^{2}/2}F=e^{-(\tau_{1}+\tau_{2}%
)D^{2}/2}F$.
\end{enumerate}
\end{proposition}

Note that (3) will not, in general, hold under the weaker assumption that
$\left\vert \tau_{1}\right\vert ,$ $\left\vert \tau_{2}\right\vert ,$ and
$\left\vert \tau_{1}+\tau_{2}\right\vert $ are all less than $1/(2\sigma).$
Indeed, even in the case $\tau_{2}=-\tau_{1},$ it is not sufficient to assume
$\left\vert \tau_{1}\right\vert =\left\vert \tau_{2}\right\vert <1/(2\sigma).$
If, for instance, $F(z)=e^{z^{2}/2},$ then after the heat flow of with
parameter $\tau=-3/4,$ we obtain $F(z,-3/4)=2e^{2z^{2}}$, which lies outside
the domain of the candidate for the inverse heat flow $\mathrm{e}^{-\frac
{3}{4}D^{2}/2}.$ (See Example \ref{ex:heat_flow_applied_to_zn_e_az_sq_bz}.)

\begin{proof}
Let us first focus on situation (S1). If $F$ is of order $\rho<2$, then by
\eqref{eq:assumption_F} it satisfies $M(r)=\mathcal{O}( \mathrm{e}^{
\varepsilon r^{2}})$. Thus, Proposition \ref{prop:heat_flow_as_convolution}
allows to estimate the growth after the heat flow as
\begin{align*}
M_{t}(r)=\sup_{|z|=r}|\mathrm{e}^{-\frac t2 D^{2}} F(z)|\lesssim\sup_{|z|=r}
\int_{\mathbb{R}}\mathrm{e}^{a |x-z|^{\rho}-x^{2}/(2|t|)}dx=\mathcal{O}(
\mathrm{e}^{2a r^{\rho}})
\end{align*}
for $\varepsilon>0$ sufficiently small and $a>\sigma$. This proves one
direction of (a).

Now, for any $s,t\in\mathbb{C}$, since $\mathrm{e}^{-\frac s 2 D^{2}}F$ is of
order $\rho<2$ we may apply Theorem
\ref{theo:entire_function_heat_flow_exists}, where by absolute convergence we
may yet again exchange limits and differentiation to obtain
\begin{align*}
\mathrm{e}^{-\frac t 2 D^{2}}\mathrm{e}^{-\frac s 2 D^{2}}F(z)  &
=\sum_{j,k=0}^{\infty}\frac1 {j!k!} \Big(-\frac t2 D^{2}\Big)^{j} \Big(-\frac
s2 D^{2}\Big)^{k} F(z)\\
&  =\sum_{n=0}^{\infty}\frac1 {n!}\sum_{k=0}^{n} \frac{n!} {(n-k)!k!}
\Big(-\frac t2 \Big)^{n-k} \Big(-\frac s 2 \Big)^{k} D^{2n}F(z)\\
&=\mathrm{e}%
^{-\frac{t+s} 2 D^{2}}F(z).
\end{align*}
In particular taking $s=-t$ shows the inverse direction of (a), that is if
$\mathrm{e}^{-\frac t 2 D^{2}}F$ would be of order $\rho<2$, then so is
$F=\mathrm{e}^{-\frac t 2 D^{2}}\mathrm{e}^{\frac t 2 D^{2}}F$.

We turn to situation (S2) for which we already argued why the order $\rho=2$
cannot decrease after the heat flow. On the other hand, again by
(\ref{eq:assumption_F}) its growth is bounded by
\begin{align*}
M_{t}(r)=  &  \sup_{|z|=r}|\mathrm{e}^{-\frac{t}{2}D^{2}}F(z)|\lesssim
\sup_{|z|=r}\int_{\mathbb{R}}\exp\big((\sigma+\varepsilon)|x-z|^{2}%
-x^{2}/(2|t|)\big)dx\\
=  &  \sup_{|z|=r}\int_{\mathbb{R}}\exp\Big((\sigma+\varepsilon)|z|^{2}%
+\frac{\big((\sigma+\varepsilon)\operatorname{Re}z\big)^{2}}{\tfrac{1}%
{2|t|}-\sigma-\varepsilon}-\big(\frac{1}{2|t|}-\sigma-\varepsilon
\big)y^{2}\Big)dy\\
=  &  c\sup_{|z|=r}\exp\Big(\frac{\sigma+\varepsilon}{1-2|t|(\sigma
+\varepsilon)}(\operatorname{Re}z)^{2}+(\sigma+\varepsilon)(\Im z)^{2}\Big)\\
=  &  \mathcal{O}\Big(\exp\big(\frac{\sigma+\varepsilon}{1-2|t|(\sigma
+\varepsilon)}r^{2}\big)\Big),
\end{align*}
where shifted to $y=x+\frac{\sigma+\varepsilon}{1/(2|t|)-\sigma-\varepsilon
}\operatorname{Re}z$. Taking logarithm(s), we obtain $\rho=2$ and
$\sigma(t)\leq\frac{\sigma}{1-2\sigma|t|}$. (c) follows from the same
arguments as in situation (S1).
\end{proof}

The next proposition states that the evolution of the function $\mathrm{e}%
^{\frac{a}{2}z^{2}+bz}F(z)$ under the heat flow is essentially the same as the
evolution of $F(z)$, up to a linear change of variables and a simple
additional factor.

\begin{proposition}
\label{prop:heat_flow_and_mult_by_exp} Let $F$ be an entire function
satisfying~\eqref{eq:assumption_F} with some finite $\sigma_{0}\geq0$. Let
$a\in\mathbb{C}$ and $\tau\in\mathbb{C}$ be such that $|\tau|<1/(2\sigma
_{0}+|a|)$. Then, for all $b\in\mathbb{C}$
\begin{equation}
\mathrm{e}^{-\tau D^{2}/2}\left[  \mathrm{e}^{\frac{a}{2}z^{2}+bz}F(z)\right]
=\frac{1}{\sqrt{1+a\tau}}\cdot\mathrm{e}^{\frac{az^{2}+2bz-\tau b^{2}%
}{2(1+a\tau)}}\cdot\mathrm{e}^{-\frac{\tau}{2(1+a\tau)}D^{2}}\left[
F(\cdot)\right]  \left(  \frac{z-b\tau}{1+a\tau}\right)
.\label{eq:prop:heat_flow_mult_by_1}%
\end{equation}

\end{proposition}

\begin{proof}
According to Proposition~\ref{prop:heat_flow_as_convolution}, if $F$ satisfies
condition~\eqref{eq:assumption_F}, then
\begin{equation}
\mathrm{e}^{-\tau D^{2}/2}F(z)=\frac{1}{\sqrt{-2\pi\tau}}\int_{-\infty
}^{\infty}F(z-x)\mathrm{e}^{x^{2}/(2\tau)}d\,x\quad\text{ if }-\frac
{1}{2\sigma_{0}}<\tau<0. \label{eq:heat_flow_as_convol_aux1}%
\end{equation}
The function $z\mapsto\mathrm{e}^{\frac{a}{2}z^{2}+bz}F(z)$ satisfies the
following analogue of condition~\eqref{eq:assumption_F}: For every
$\varepsilon>0$ we have
\[
\max_{|z|=r}|\mathrm{e}^{\frac{a}{2}z^{2}+bz}F(z)|\leq\mathrm{e}^{(\sigma
_{0}+\frac{1}{2}|a|+\varepsilon)r^{2}},
\]
provided $r>r(\varepsilon)$ is sufficiently large. It follows from
Theorem~\ref{theo:entire_function_heat_flow_exists} that \linebreak$\mathrm{e}^{-\tau
D^{2}/2}[\mathrm{e}^{\frac{a}{2}z^{2}+bz}F(z)]$ defines an analytic function
of the variables $(z,\tau)$ in the range $z\in\mathbb{C}$, $|\tau
|<1/(2\sigma_{0}+|a|)$.

Let us first prove~\eqref{eq:prop:heat_flow_mult_by_1} assuming that $a=1,b=0$
and $z\in\mathbb{R}$, $\tau\in\mathbb{R}$ satisfies $-1/(2\sigma_{0}%
+1)<\tau<0$. Applying~\eqref{eq:heat_flow_as_convol_aux1} to the function
$z\mapsto\mathrm{e}^{\frac{1}{2}z^{2}}F(z)$ yields
\begin{align*}
\mathrm{e}^{-\tau D^{2}/2}\left[  \mathrm{e}^{\frac{1}{2}z^{2}}F(z)\right]
&  =\frac{1}{\sqrt{-2\pi\tau}}\int_{-\infty}^{\infty}F(z-x)\mathrm{e}%
^{\frac{1}{2}(z-x)^{2}}\mathrm{e}^{x^{2}/(2\tau)}\,dx\\
&  =\frac{1}{\sqrt{-2\pi\tau}}\mathrm{e}^{\frac{1}{2}z^{2}}\int_{-\infty
}^{\infty}F(z-x)\mathrm{e}^{\frac{x^{2}}{2}\cdot\frac{1+\tau}{\tau}-xz}\,dx\\
&  =\frac{1}{\sqrt{-2\pi\tau}}\mathrm{e}^{\frac{1}{2}z^{2}-\frac{\tau z^{2}%
}{2(1+\tau)}}\int_{-\infty}^{\infty}F(z-x)\mathrm{e}^{\frac{1}{2\tau_{\ast}%
}\cdot\left(  x-\frac{\tau z}{1+\tau}\right)  ^{2}}\,dx,
\end{align*}
where we defined $\tau_{\ast}:=\tau/(1+\tau)$. One easily checks that
$-1/(2\sigma_{0})<\tau_{\ast}<0$. After the substitution $y:=x-\frac{\tau
z}{1+\tau}$ (observe that $y$ takes real values because $z$ is real), we get
\begin{align*}
\mathrm{e}^{-\tau D^{2}/2}\left[  \mathrm{e}^{\frac{1}{2}z^{2}}F(z)\right]
&  =\frac{1}{\sqrt{-2\pi\tau}}\mathrm{e}^{\frac{z^{2}}{2(1+\tau)}}%
\int_{-\infty}^{\infty}F\left(  z-y-\frac{\tau z}{1+\tau}\right)
\mathrm{e}^{y^{2}/(2\tau_{\ast})}\,dy\\
&  =\frac{1}{\sqrt{1+\tau}}\cdot\mathrm{e}^{\frac{z^{2}}{2(1+\tau)}}\cdot
\frac{1}{\sqrt{-2\pi\tau_{\ast}}}\int_{-\infty}^{\infty}F\left(  \frac
{z}{1+\tau}-y\right)  \mathrm{e}^{y^{2}/(2\tau_{\ast})}\,dy\\
&  =\frac{1}{\sqrt{1+\tau}}\cdot\mathrm{e}^{\frac{z^{2}}{2(1+\tau)}}%
\cdot\mathrm{e}^{-\frac{\tau_{\ast}}{2}D^{2}}\left[  F(\cdot)\right]  \left(
\frac{z}{1+\tau}\right)  ,
\end{align*}
where in the last step we used~\eqref{eq:heat_flow_as_convol_aux1}. This proves~\eqref{eq:prop:heat_flow_mult_by_1}.

Recall that $(z,\tau)\mapsto\mathrm{e}^{-\tau D^{2}/2}[\mathrm{e}^{\frac{1}%
{2}z^{2}}F(z)]$ is an analytic function of two variables on $\{(z,\tau
)\in\mathbb{C}^{2}:z\in\mathbb{C},|\tau|<1/(2\sigma_{0})\}$. The right-hand
side of~\eqref{eq:prop:heat_flow_mult_by_1} also defines an analytic function
on the same domain, and both functions coincide if $z\in\mathbb{R}$ and
$\tau<0$. By the uniqueness of analytic continuation both functions coincide
on the whole domain of definition, which
proves~\eqref{eq:prop:heat_flow_mult_by_1} for $a=1,b=0$.

To prove~\eqref{eq:prop:heat_flow_mult_by_1} for arbitrary $a,b\in\mathbb{C}$,
$a\neq0$, we use the behavior under linear changes of variables to $\tilde
{z}=\sqrt{a}z+b/\sqrt{a}$ such that $\frac{a}{2}z^{2}+bz=\frac{1}{2}\tilde
{z}^{2}-\frac{b^{2}}{2a}$. For $\tilde{F}(\tilde{z})=F(\tilde{z}/\sqrt
{a}-b/a)=F(z)$ we apply the previous proof (for $a\tau$ replacing $\tau$) and
Lemma \ref{lem:heat_flow_properties1} to obtain
\begin{align*}
\mathrm{e}^{-\tau D^{2}/2}\left[  \mathrm{e}^{\frac{a}{2}z^{2}+bz}F(z)\right]
&  =\mathrm{e}^{-\frac{b^{2}}{2a}}\mathrm{e}^{-\tau D^{2}/2}\left[
\mathrm{e}^{\frac{1}{2}\tilde{z}^{2}}\tilde{F}(\tilde{z})\right]
=\mathrm{e}^{-\frac{b^{2}}{2a}}\mathrm{e}^{-\frac{a\tau}{2}D^{2}}\left[
\mathrm{e}^{\frac{1}{2}z^{2}}\tilde{F}(z)\right]  (\tilde{z})\\
&  =\mathrm{e}^{-\frac{b^{2}}{2a}}\frac{1}{\sqrt{1+a\tau}}\cdot\mathrm{e}%
^{\frac{\tilde{z}^{2}}{2(1+a\tau)}}\cdot\mathrm{e}^{-\frac{a\tau}{2(1+a\tau
)}D^{2}}\left[  \tilde{F}\right]  \left(  \frac{\tilde{z}}{1+a\tau}\right) \\
&  =\frac{1}{\sqrt{1+a\tau}}\cdot\mathrm{e}^{\frac{az^{2}+2bz-\tau b^{2}%
}{2(1+a\tau)}}\mathrm{e}^{-\frac{\tau}{2(1+a\tau)}D^{2}}\left[  F\right]
\left(  \frac{z-b\tau}{1+a\tau}\right).
\end{align*}
Therefore, \eqref{eq:prop:heat_flow_mult_by_1} holds.
\end{proof}

%Lastly, \eqref{eq:prop:heat_flow_mult_by_2} follows again from Lemma \ref{lem:heat_flow_properties1}.

\begin{example}
\label{ex:heat_flow_applied_to_zn_e_az_sq_bz} Taking $F(z)=1$ in
Proposition~\ref{prop:heat_flow_and_mult_by_exp} we obtain the following
result. Let $a,b\in\mathbb{C}$ and $\tau\in\mathbb{C}$ be such that
$|a\tau|<1$. Then,
\begin{equation}
\mathrm{e}^{-\tau D^{2}/2}\left[  \mathrm{e}^{\frac{a}{2}z^{2}+bz}\right]
=\frac{1}{\sqrt{1+a\tau}}\cdot\exp\left\{  \frac{az^{2}+2bz-\tau b^{2}%
}{2(1+a\tau)}\right\}  . \label{eq:heat_flow_exponential_new}%
\end{equation}
More generally, under the same assumptions and for every $n\in\mathbb{N}_{0}$,
taking $F(z)=z^{n}$ and recalling~\eqref{eq:eq:hermite_exp_on_z^n_with_sigma}
yields
%Let $a,b\in \mathbb{C}$ and $t\in \mathbb{C}$ be such that $|at|<1$.  Then, for every $n\in \mathbb{N}_0$ we have%
\begin{align*}
&\mathrm{e}^{-\tau D^{2}/2}[z^{n}\mathrm{e}^{\frac{a}{2}z^{2}+bz}]\\&=\frac
{1}{\sqrt{1+a\tau}}\cdot\exp\left\{  \frac{az^{2}+2bz-\tau b^{2}}{2(1+a\tau
)}\right\}  \cdot\left(  \frac{\tau}{1+a\tau}\right)  ^{n/2}\operatorname{He}%
_{n}\left(  \frac{z-b\tau}{\sqrt{\tau(1+a\tau)}}\right)  .
\label{eq:heat_flow_exponential_new_with_zn}%
\end{align*}

\end{example}

\begin{example}
[Mehler's formula]\label{mehler.ex}Let us
use~\eqref{eq:heat_flow_exponential_new} to derive the following well-known
Mehler formula (see, e.g., \cite[18.18.28]{NIST}):
\begin{equation}
\sum_{n=0}^{\infty}\frac{\rho^{n}}{n!}\operatorname{He}_{n}%
(x)\operatorname{He}_{n}(y)=\frac{1}{\sqrt{1-\rho^{2}}}\exp\left\{
-\frac{\rho^{2}(x^{2}+y^{2})-2\rho xy}{2(1-\rho^{2})}\right\}  ,
\label{eq:mehler_formula}%
\end{equation}
for all $x,y\in\mathbb{C}$ and $\rho\in\mathbb{D}$. Using the generating
function for the Hermite polynomials given by
%\begin{equation}\label{eq:generating_funct_Hermite_poly}%
\[
F(u)=\mathrm{e}^{xu-\frac{1}{2}u^{2}}=\sum_{n=0}^{\infty}\operatorname{He}%
_{n}(x)\frac{u^{n}}{n!},
\]
for fixed $x\in\mathbb{C}$, and using the definition of the heat flow given
in~\eqref{eq:heat_flow_def_2}, we get
\begin{align*}
\sum_{n=0}^{\infty}\frac{\rho^{n}}{n!}\operatorname{He}_{n}%
(x)\operatorname{He}_{n}(y)&=\mathrm{e}^{-\frac{1}{2}\left(  \frac{d\,}%
{d\,y}\right)  ^{2}}\sum_{n=0}^{\infty}\frac{\rho^{n}}{n!}\operatorname{He}%
_{n}(x)y^{n}=\mathrm{e}^{-\frac{1}{2}\left(  \frac{d\,}{d\,y}\right)  ^{2}%
}\mathrm{e}^{\rho xy-\frac{1}{2}\rho^{2}y^{2}}\\
&=\mathrm{e}^{-\frac{\rho}%
{2}D^{2}}\big[F\big](\rho y).
\end{align*}
Applying~\eqref{eq:heat_flow_exponential_new}
yields~\eqref{eq:mehler_formula}. From this perspective, Mehler's formula can
be seen as the heat evolved generating function after \textquotedblleft
time\textquotedblright\ $\rho$.
\end{example}

\subsection{Evolution of the zeros}

We start with an elementary application of the implicit function theorem to
the zeros of $(e^{-\tau D^{2}/2}F)(z).$

\begin{proposition}
\label{diffZero.prop}Suppose $F$ is an entire function of order $\rho<2$ or
order $\rho=2$ and finite type $\sigma$ and let $F(z,\tau)=(e^{-\tau D^{2}%
/2}F)(z)$, where we assume $\left\vert \tau\right\vert <1/(2\sigma)$ when
$\rho=2.$ Suppose $z_{0}$ is a simple zero of $F(z,\tau_{0}).$ Then we can
find a unique holomorphic function $z(\tau),$ defined for $\tau$ in some disk
centered at $\tau_{0}$, such that $z(\tau_{0})=z_{0}$ and such that
$F(z(\tau),\tau)=0$. Furthermore, we have%
\begin{equation}
z^{\prime}(\tau)=\frac{1}{2}\frac{\partial_{1}^{2}F(z(\tau),\tau)}%
{\partial_{1}F(z(\tau),\tau)}. \label{zPrime1}%
\end{equation}
We can also write this result as%
\begin{equation}
z^{\prime}(\tau)=\left.  \left(  \frac{\partial_{1}F(z(\tau),\tau)}%
{F(z(\tau),\tau)}-\frac{1}{z-z(\tau)}\right)  \right\vert _{z=z(\tau)},
\label{zPrime2}%
\end{equation}
where part of the statement is that the right-hand side of (\ref{zPrime2}) has
a removable singularity at $z=z(\tau).$
\end{proposition}

\begin{proof}
Since $z$ is a simple zero, $F^{\prime}(z)$ is nonzero and thus $\partial
_{1}F(z,\tau)$ is nonzero at $(z,0).$ Thus, by the holomorphic version of the
implicit function theorem, we can solve the equation $F(z,\tau)=0$ uniquely
for $z$ as a holomorphic function of $\tau$ near $\tau=0,$ with $z(0)=z.$

Then by differentiating the relation $F(z(\tau),\tau)$ with respect to $\tau,$
we find that
\[
z^{\prime}(\tau)=-\frac{\partial_{2}F}{\partial_{1}F}(z(\tau),\tau).
\]
Then since $\partial_{2}F(z,\tau)=-\frac{1}{2}\partial_{2}F(z,\tau)$, we
obtain the claimed formula (\ref{zPrime1}) for $z^{\prime}(\tau).$ It is then
an elementary calculation with power series to verify that if a holomorphic
function $f$ has a simple zero at $c,$ then%
\[
\frac{f^{\prime\prime}(c)}{f^{\prime}(c)}=2\left.  \left(  \frac{f^{\prime
}(z)}{f(z)}-\frac{1}{z-c}\right)  \right\vert _{z=c},
\]
from which (\ref{zPrime2}) follows.
\end{proof}

We now record a well-known result about the evolution of the zeros in the
polynomial case.

\begin{proposition}
\label{prop:PolyEvolve}Suppose that, for some $\sigma\in\mathbb{C},$ the zeros
of $p$ are distinct. Then for all $\tau$ in a neighborhood of $\sigma,$ it is
possible to order the zeros of $e^{-\tau D^{2}/2}p$ as $z_{1}(\tau
),\ldots,z_{N}(\tau)$ so that each $z_{j}(\tau)$ depends holomorphically on
$\tau$ and so that the collection $\{z_{j}(\tau)\}_{j=1}^{N}$ satisfies the
following system of holomorphic differential equations:%
\begin{equation}
\frac{dz_{j}(\tau)}{d\tau}=\sum_{k:~k\neq j}\frac{1}{z_{j}(\tau)-z_{k}(\tau)}.
\label{theODE}%
\end{equation}
The paths $z_{j}(\tau)$ then satisfy%
\begin{equation}
\frac{d^{2}z_{j}(\tau)}{d\tau^{2}}=-2\sum_{k:~k\neq j}\frac{1}{(z_{j}%
(\tau)-z_{k}(\tau))^{3}}. \label{secondDeriv}%
\end{equation}

\end{proposition}

The equation (\ref{theODE}) can easily be obtained from (\ref{zPrime2}); then
(\ref{secondDeriv}) follows by a routine calculation. The result in
(\ref{theODE}) is discussed on Terry Tao's blog \cite{tao_blog1,tao_blog2} and
dates back at least to the work of the Chudnovsky brothers \cite{Chudnovsky}.
See also Csordas, Smith, and Varga \cite{csordas_smith_varga} and Hall--Ho
\cite{hallho}.

The formula (\ref{secondDeriv}) is the equation of motion for the
\textbf{rational Calogero--Moser system}. (Take $\omega=0$ and $g^{2}=-1/N$ in
the notation of \cite[Eq. (3)]{Cal}.) It follows that solutions to
(\ref{theODE}) are special cases of solutions to the rational Calogero--Moser
system, in which the initial velocities are chosen to satisfy (\ref{theODE})
at $\tau=0.$

If we replace the polynomial $p$ with an entire function $F,$ the result of
Proposition \ref{prop:PolyEvolve} will not necessarily hold, even if $F$ has
only finitely many zeros. After all, there can be many entire functions with
the same finite set of zeros and the zeros of these different functions will
not evolve in the same way under the heat flow. Furthermore, for a general
entire function, the sum on the right-hand side of (\ref{theODE}) will not be convergent.

We will derive versions of (\ref{theODE}) and (\ref{secondDeriv}) for entire
functions in Section 4. Remarkably, we will find that the second-derivative
formula (\ref{secondDeriv}) holds exactly as written---as a convergent
infinite sum---for entire functions of order $\rho<2$ \textit{and} for entire
functions of order $\rho=2$ and finite type. This result holds even though a
function of this sort is not uniquely determined by its zeros. The
first-derivative formula (\ref{theODE}), however, will need to be modified appropriately.

\subsection{Examples}

We now consider several explicit examples of how an entire function (and its
zeros) evolve under the heat flow.

\begin{example}
Consider the function
\[
F(z)=e^{a_{0}+a_{1}z}\sin(\pi z)=\frac{1}{2}e^{a_{0}+(a_{1}+i\pi)z}-\frac
{1}{2}e^{a_{0}+(a_{1}-i\pi)z}%
\]
which has a simple zero at each integer. Then by Example
\ref{ex:heat_flow_applied_to_zn_e_az_sq_bz}, we have
\begin{align*}
e^{-\tau D^{2}/2}F(z)  &  =\frac{1}{2}e^{a_{0}+(a_{1}+\pi i)z-\frac{\tau}%
{2}(a_{1}+\pi i)^{2}}-\frac{1}{2}e^{a_{0}+(a_{1}-\pi i)z-\frac{\tau}{2}%
(a_{1}-\pi i)^{2}}\\
&  =e^{a_{0}+a_{1}z-\frac{\tau}{2}a_{1}^{2}+\frac{\tau}{2}\pi^{2}}\sin
(\pi(z-\tau a_{1})),
\end{align*}
which has zeros at $n+\tau a_{1},$ $n\in\mathbb{Z}.$
\end{example}

\begin{example}
Let $F(z)=\sin(\pi z^{2}),$ so that the zeros are initially located on the
real and imaginary axes at the points $z_{n}^{+}=\sqrt{n}$, $z_{n}^{-}%
=-\sqrt{n}$, for $n\in\mathbb{Z}$. Note that there is a zero of multiplicity
$2$ at $0$ which we write as $z_{0}^{+}=z_{0}^{-}=0$.
By~\eqref{eq:heat_flow_exponential_new}, for every $\tau\in\mathbb{C}$ with
$|2\pi\tau|<1$ we have
\[
e^{-\tau D^{2}/2}F(z)=\frac{1}{2}\frac{\exp\left\{  \frac{\pi iz^{2}}{1+2\pi
i\tau}\right\}  }{\sqrt{1+2\pi i\tau}}-\frac{1}{2}\frac{\exp\left\{
-\frac{\pi iz^{2}}{1-2\pi i\tau}\right\}  }{\sqrt{1-2\pi i\tau}}.
\]
Equating this to zero yields
\[
\exp\left\{  \frac{2\pi iz^{2}}{1+4\pi^{2}\tau^{2}}\right\}  =\sqrt
{\frac{1+2\pi i\tau}{1-2\pi i\tau}}.
\]

The solutions, defined for $|\tau|<1/(2\pi)$, take the form
\begin{align}
z_{n}^{\pm}(\tau)  &  =\pm\sqrt{\left(  \frac{1}{4\pi i}\log\left(
\frac{1+2\pi i\tau}{1-2\pi i\tau}\right)  +n\right)  \cdot(1+4\pi^{2}\tau
^{2})}\nonumber\\
&  =\pm\sqrt{\left(  \frac{\arctan(2\pi\tau)}{2\pi}+n\right)  \cdot(1+4\pi
^{2}\tau^{2})}, \label{eq:zeros_at_square_roots_heat_flow}%
\end{align}
for all $n\in\mathbb{Z}$, where we use the principal branch of $\arctan z$
satisfying $\arctan0=0$. Note that $z_{0}^{+}(\tau)$ and $z_{0}^{-}(\tau)$
denote two different trajectories (which meet at $\tau=0$). In a similar way,
one can compute the action of the heat flow on the function $G(z)=\cos(\pi
z^{2})$ with zeros at $z_{n}^{\pm}=\pm\sqrt{n+1/2}$, $n\in\mathbb{Z}$. The
final result is the same as in~\eqref{eq:zeros_at_square_roots_heat_flow}
with $n$ replaced by $n+1/2$.
\end{example}

\begin{example}
Let us consider the case when the initial conditions $z_{j}(0)$ form a lattice
in $\mathbb{C}$. To this end, take $F(z)$ to be the Jacobi theta function
\[
\vartheta(z;\sigma):=\sum_{n=-\infty}^{+\infty}q^{n^{2}}e^{2\pi izn},\qquad
q=e^{\pi i\sigma},
\]
where $z\in\mathbb{C}$ and $\sigma\in\mathbb{C}$ is a parameter satisfying
$\operatorname{Im}\tau>0$; see~\cite[Chapter XXII]{whittaker_watson_book}
or~\cite[Chapter~1]{lawden_book}. The zeros of $z\mapsto\vartheta(z;\tau)$
are simple and located at $m+n\sigma+\frac{1}{2}+\frac{\sigma}{2}$, for
$m,n\in\mathbb{Z}$. Thus, the zeros form a lattice in $\mathbb{C}$ with
periods $1$ and $\sigma$. Using the quasi-periodicity relation
\[
\vartheta(z+m+n\sigma;\sigma)=e^{-\pi in^{2}\sigma-2\pi inz}\vartheta
(z;\sigma),\qquad m,n\in\mathbb{Z},
\]
one checks that $z\mapsto\vartheta(z;\sigma)$ is an entire function of order
$2$ and type $\pi/\operatorname{Im}\sigma$.
Using~\eqref{eq:heat_flow_exponential_new} one checks that
\[
e^{-\tau D^{2}/2}\vartheta(z;\sigma)=\sum_{n=-\infty}^{+\infty}q^{n^{2}%
}e^{2\pi inz+2\pi^{2}n^{2}\tau}=\vartheta(z;\sigma-2\pi i\tau),
\]
provided $\left\vert \tau\right\vert <(\operatorname{Im}\tau)/(2\pi)$. It
follows that the zeros of $e^{-\tau D^{2}/2}\vartheta(z;\sigma)$ form a
lattice with periods $1$ and $\sigma-2\pi i\tau$.
\end{example}

\section{Gaussian analytic function undergoing the heat flow}

\subsection{The main result}
In this section, we prove our main result Theorem \ref{theo:GAF_invariant_generalization_zeros_introduction}, which we repeat below. 
Let $\tau\in\mathbb{D}$ and $F$ be an
entire function of order $\rho\leq2$, where if $\rho=2$ we assume $F$ has type
at most $1/2.$ Let $V_{\tau}F$ be the entire function given by%
\begin{equation}
\lbrack V_{\tau}F](z)=\left(  1-\left\vert \tau\right\vert ^{2}\right)
^{1/4}~e^{\bar\tau z^{2}/2}\left(  e^{-\tau D^{2}/2}F(\cdot)\right)  \left(
z\sqrt{1-\left\vert \tau\right\vert ^{2}}\right)  . \label{VtauDef}%
\end{equation}

%The next theorem states that the distribution of the plane GAF stays invariant
%under the heat flow, up to a simple rescaling. 
Since the plane GAF $G$ is an entire function of order $\rho=2$ and type $\sigma=1/2$, it follows
from Theorem~\ref{theo:entire_function_heat_flow_exists} that $\exp\{-\tau
D^{2}/2\}G$ exists (as an entire function) for all $\tau\in\mathbb{D}$.%, where $\mathbb{D}=\{\tau\in\mathbb{C}:|\tau|<1\}$ denotes the unit disk.

\begin{theorem}
\label{theo:GAF_invariant} 
If $G$ is a GAF, then $V_{\tau}G$ is also a GAF. %That is, $[V_{\tau}G](z)$ has the same distribution as $G(z).$
\end{theorem}

Note that the factor of $\left(  1-\left\vert \tau\right\vert ^{2}\right)
^{1/4}~e^{\tau z^{2}/2}$ on the right-hand side of (\ref{VtauDef}) does not
affect the zeros of $e^{-\tau D^{2}/2}G.$ It follows that the \textit{zeros} of
$e^{-\tau D^{2}/2}G$ have the same distribution as the zeros of $G,$ up to a
scaling by a factor of $\sqrt{1-\left\vert \tau\right\vert ^{2}},$ leading to the following result:
 $$\mathcal Z (G)\stackrel{d}{=}\frac{\mathcal Z (e^{-\tau D^2/2}G)}{\sqrt{1-\left\vert \tau\right\vert ^{2}}}$$
%
%and let
%$\{z_{j}\}$ and $\{w_{j}\}$ be the zeros of $G$ and $e^{\tau D^{2}/2}G,$
%respectively. Then $\{z_{j}\}$ and
%\[
%\left\{\frac{w_{j}}{\sqrt{1-\left\vert \tau\right\vert ^{2}}}\right\}
%\]
%have the same distribution as random collections of points in the plane.

\begin{remark}
In particular $G$ and
$
\big(  e^{-\tau
D^{2}/2}G\big)  \big(  z\sqrt{1-\vert \tau\vert ^{2}%
}\big)
$
have equal mean intensity of zeros. Thus, the equality in
distribution up to some deterministic analytic function $\varphi$ follows from
the \textquotedblleft Calabi rigidity\textquotedblright\ result of Sodin
\cite[Theorem 2.5.2]{hough_etal_book}. In Theorem \ref{theo:GAF_invariant}
however the weight function $\varphi$ is made explicit and turns out to be a Gaussian.
\end{remark}

\begin{proof}[Proof of Theorem \ref{theo:GAF_invariant}]
For $\tau\in\mathbb{D}$ consider the random entire function
\[
G_{\tau}(z):=\mathrm{e}^{-\tau D^{2}/2}G(z)=\sum_{n=0}^{\infty}\frac{\xi_{n}%
}{\sqrt{n!}}(\sqrt{\tau})^{n}\operatorname{He}_{n}\left(  \frac{z}{\sqrt{\tau
}}\right)  .
\]
Since the finite-dimensional distributions of the processes $(G_{\tau
}(z))_{z\in\mathbb{C}}$ and $(H_{\tau}(z))_{z\in\mathbb{C}}$ are multivariate
complex Gaussian, it suffices to verify the equality of the covariance
functions of these processes. In fact, we shall compute even the covariance
functions of the processes $(G_{\tau}(z))_{z\in\mathbb{C},\tau\in\mathbb{D}}$
and $(H_{\tau}(z))_{z\in\mathbb{C},\tau\in\mathbb{D}}$ (which, as will turn
out, are not equal). Take some $z,w\in\mathbb{C}$ and $\tau,\sigma
\in\mathbb{D}$. Then,
\[
\mathbb{E}\left[  G_{\tau}(z)\overline{G_{\sigma}(w)}\right]  =\sum
_{n=0}^{\infty}\frac{(\sqrt{\tau}\sqrt{\bar{\sigma}})^{n}}{n!}%
\operatorname{He}_{n}\left(  \frac{z}{\sqrt{\tau}}\right)  \operatorname{He}%
_{n}\left(  \frac{\overline{w}}{\sqrt{\bar{\sigma}}}\right)  ,
\]
where we used independence of the coefficients $\xi_{n}$. By the Mehler
formula~\eqref{eq:mehler_formula} it follows that
\begin{align}
\mathbb{E}\left[  G_{\tau}(z)\overline{G_{\sigma}(w)}\right]   &  =\frac
{1}{\sqrt{1-\tau\bar{\sigma}}}\exp\left\{  \frac{\tau\bar{\sigma}\left(
-\frac{z^{2}}{\tau}-\frac{\overline{w}^{2}}{\bar{\sigma}}\right)
+2z\overline{w}}{2(1-\tau\bar{\sigma})}\right\} \nonumber\\
&  =\frac{1}{\sqrt{1-\tau\bar{\sigma}}}\exp\left\{  -\frac{z^{2}\bar{\sigma
}+\overline{w}^{2}\tau}{2(1-\tau\bar{\sigma})}\right\}  \exp\left\{
\frac{z\overline{w}}{1-\tau\bar{\sigma}}\right\}  . \label{eq:cov_G_t_aux}%
\end{align}
Putting $\sigma=\tau$ and using the definition (\ref{VtauDef})\ of $V_{\tau
}G,$ we easily find that
\[
\mathbb{E}\left[  (V_{\tau}G)(z)\overline{(V_{\tau}G)(w)}\right]  =e^{z\bar
{w}},
\]
which coincides with $\mathbb{E}\left[  G(z)\overline{G(w)}\right]  .$
\end{proof}

We can give a second proof of Theorem \ref{theo:GAF_invariant} as follows. We
have already noted that the law of the GAF is the standard Gaussian measure
based on the Segal--Bargmann space $\mathcal{B}$ in (\ref{eq:SBdef}), in the
sense that the functions $z^{n}/\sqrt{n!}$ in (\ref{eq:GAFdef}) form an
orthonormal basis for $\mathcal{B}.$ Meanwhile, the restriction of $V_{\tau}$
to $\mathcal{B}$ is a unitary map, part of the family of unitary operators
making up the metaplectic representation. (See Section \ref{meta.sec}.)
Theorem \ref{theo:GAF_invariant} the follows from the unitarity of $V_{\tau}$
and the following elementary result.

\begin{proposition}
Suppose $U$ is a unitary map of $\mathcal{B}$ to itself and let $f_{n}%
(z)=U(z^{n}/\sqrt{n!}).$ Assume that the series
\[
G_{U}(z)=\sum_{n=0}^{\infty}\xi_{n}f_{n}(z)
\]
converges locally uniformly in $z$ almost surely. Then $G_{U}$ is again a GAF;
i.e., it has the same distribution as $G.$
\end{proposition}

The assumption on $U$ should not really be necessary, but it is easily
verified in the case $U=V_{\tau}$ (by the estimates in the proof of Theorem
\ref{theo:entire_function_heat_flow_exists}) and assuming it saves us some
technical difficulty.

\begin{proof}
Since the $\xi_{n}$'s are independent with mean zero and variance 1, we may
compute the covariance of $G_{U}(z)$ as%
\begin{align*}
\mathbb{E}\{G_{U}(z)\overline{G_{U}(w)}\}  &  =\sum_{n=0}^{\infty}%
f_{n}(z)\overline{f_{n}(w)}  =e^{z\bar{w}},
\end{align*}
giving the same covariance as for $G$ itself. Here we have used the formula
for the reproducing kernel of the Segal--Bargmann space, which may be computed
using an arbitrary orthonormal basis of $\mathcal{B}.$ (See Theorem 2.4 in
\cite{hall_holomorphic_methods}.)
\end{proof}

\begin{remark}
Taking $z=w=0$ in~\eqref{eq:cov_G_t_aux} implies that the covariance function
of the complex Gaussian process $(G_{\tau}(0))_{\tau\in\mathbb{D}}$ is given
by
\[
\mathbb{E}\left[  G_{\tau}(0)\overline{G_{s}(0)}\right]  =\frac{1}%
{\sqrt{1-\tau\overline{s}}},\qquad\tau,s\in\mathbb{D}.
\]
This random analytic function corresponds to the special case $L=1/2$ of the
family studied in Section~2.3 of~\cite{hough_etal_book}; see Equation~(2.3.6)
there. It is also shown in \cite[Lemma 2.3.4]{hough_etal_book} that this
analytic function does not extend to a larger domain than $\mathbb{D}$. The
law of the zero set of this process, that is the law of the point process
$\{\tau\in\mathbb{D}:\exp\{-\frac{\tau}{2}D^{2}\}[G](0)=0\}$ is invariant with
respect to the natural action of the group $SU(1,1)$ on $\mathbb{D}$; see
Proposition 2.3.4 in~\cite{hough_etal_book}. More general facts in this
direction will be established in Section~\ref{sec:hyperbolic_invariance}.
\end{remark}

\subsection{Approximation by polynomials}\label{subsec:approx}

Recall the definition of the GAF in (\ref{eq:GAFdef}). We introduce the
\textbf{truncated GAF} $W_{N}$ by
\[
W_{N}(z):=\sum_{n=0}^{N}\xi_{n}\frac{z^{n}}{\sqrt{n!}},
\]
where we use the same values of the random variables $\{\xi_{n}\}$ to define
$W_{N}$ and $G.$ Then $W_{N}$ is a \textbf{Weyl polynomial}, scaled so that
the zeros cluster asymptotically into a disk of radius $\sqrt{N}.$ In this
subsection, we provide an elementary result that the zeros of $e^{-\tau
D^{2}/2}G$ can be approximated by the zeros of $e^{-\tau D^{2}/2}W_{N}.$ (All
the simulations shown in the paper are made using this approximation.)\

Now, the first two authors have developed a general conjecture \cite{hallho}
concerning the behavior of zeros of polynomials under the heat flow. (In
\cite{hallho}, the heat flow is scaled as $e^{\tau D^{2}/(2N)},$ where $N$ is
the degree of the polynomials involved, so the formulas there need to be
adjusted slightly to match the sign convention and normalization of the heat
flow in the present paper.) Specifically, if the empirical measure of the
zeros of a high-degree polynomial $p$ resembles a measure $\mu,$ the
conjecture says that the zeros tend to evolve along straight lines:%
\[
z_{j}(\tau)\approx z_{j}(0)-\tau V(z_{j}(0)),
\]
where the \textquotedblleft velocity function\textquotedblright\ $V$ is the
Cauchy transform of the measure $\mu$:%
\[
V(z)=-\int_{\mathbb{C}}\frac{1}{z-w}~d\mu(w).
\]
In the case of the Weyl polynomials $W_{N},$ the measure $\mu$ would be
$\frac{1}{\pi}$ times the standard area measure on the disk of radius
$\sqrt{N}$ centered at 0, in which case%
\[
V(z)=-\bar{z}%
\]
for all $z$ in the disk. Thus, the approximate behavior (\ref{eq:zjTraj}) of
the zeros of the GAF is consistent with the conjecture in \cite{hallho}.

In a forthcoming paper, the authors of the present paper will verify this conjecture for a wide class of random
polynomials at the bulk level. However, the Gaussian nature of the GAF allows to even study the evolution of individual zeros, as we will do in the next section.

\begin{proposition}
Let $G$ be a GAF and let $W_{N}$ be its truncation. Fix a compact set
$K\subset\mathbb{C}$ and a complex number $\tau\in\mathbb{D}.$ Then the zeros
of $e^{-\tau D^{2}/2}G$ in $K$ are uniformly approximated by the zeros of
$e^{-\tau D^{2}/2}W_{N}$, in the following sense. Let $z_{1},\ldots,z_{k}$ be
the zeros of $e^{\tau D^{2}/2}G$ in $K$ and let $m_{1},\ldots,m_{k}$ be the
corresponding multiplicities. Choose $\varepsilon_{1},\ldots,\varepsilon_{k}$
small enough that the balls $B(z_{j},\varepsilon_{j})$ around $z_{j}$ are
disjoint. Then for all sufficiently large $N,$ (1) for each $j,$ there are
exactly $m_{j}$ zeros of $e^{\tau D^{2}/2}W_{N}$ (counted with their
multiplicities) in $B(z_{j},\varepsilon_{j}),$ and (2) every zero of $e^{\tau
D^{2}/2}W_{N}$ in $K$ belongs to one of the $B(z_{j},\varepsilon_{j}).$
\end{proposition}

Note that the zeros of $e^{-\tau D^{2}/2}W_{N}$ in $B(z_{j},\varepsilon_{j})$
need not be in $K.$ We may think of applying the proposition with $K$ being
closed disk of large radius centered at the origin.

\begin{proof}
The result is a straightforward application of Rouch\'{e}'s theorem (or the
argument principle). Compare Lemma 2.2 in~\cite{shirai}.
\end{proof}

\subsection{The evolution of individual zeros}

In this section, we provide the proof of Theorem \ref{theo:distributionZeros}.
We make use of the \textquotedblleft unitarized translation\textquotedblright%
\ operators $T_{a},$ $a\in\mathbb{C},$ acting on the space of all entire
functions and given by%
\[
(T_{a}F)(z)\mapsto e^{-\left\vert a\right\vert ^{2}/2}e^{\bar{a}z}F(z-a).
\]
For each $a,$ $T_{a}$ is a unitary map of the Segal--Bargmann space to itself
\cite[Theorem 4.2]{hall_holomorphic_methods}. Furthermore, if $G$ is a GAF,
then $T_{a}G$ is also a GAF \cite[proof of Proposition 2.3.4]{hough_etal_book}.

\begin{lemma}
\label{lem:shiftZeros}Suppose $F$ is an entire function of order $\rho<2$ or
an entire function of order 2 and finite type $\sigma.$ Fix $\tau\in
\mathbb{C},$ where if $F$ has order 2 and type $\sigma,$ we assume $\left\vert
\tau\right\vert <1/(2\sigma)$ and fix $a\in\mathbb{C}.$ Then $e^{-\tau D^{2}%
/2}F$ has a zero at $z$ if and only if $e^{-\tau D^{2}/2}(T_{a}F)$ has a zero
at $z+a+\tau\bar{a}.$
\end{lemma}

\begin{proof}
Using Proposition \ref{prop:heat_flow_and_mult_by_exp} with $a=0,$ we find
that
\[
e^{-\tau D^{2}/2}[e^{\bar{a}z}F(z)]=e^{-\tau\bar{a}^{2}/2}e^{\bar{a}%
z}(e^{-\tau D^{2}/2}F)(z-\tau\bar{a}).
\]
Using this result and that (Lemma \ref{lem:heat_flow_properties1}) the heat
operator commutes with translations, we find that%
\begin{align*}
e^{-\tau D^{2}/2}(T_{a}F)(z)  &  =e^{-\tau D^{2}/2}\left[  e^{-\left\vert
a\right\vert ^{2}/2}e^{\bar{a}z}F(z-a)\right] \\
&  =e^{\left\vert a\right\vert ^{2}/2}e^{-\tau D^{2}/2}\left[  e^{\bar
{a}(z-a)}F(z-a)\right] \\
&  =e^{\left\vert a\right\vert ^{2}/2}e^{-\tau D^{2}/2}[e^{\bar{a}%
z}F(z)](z-a)\\
&  =e^{\left\vert a\right\vert ^{2}/2}e^{-\tau\bar{a}^{2}/2}e^{\bar{a}%
z}[e^{-\tau D^{2}/2}F](z-a-\tau\bar{a}).
\end{align*}
Thus, $e^{-\tau D^{2}/2}(T_{a}F)$ has zero at $z$ if and only if $e^{-\tau
D^{2}/2}F$ has a zero at $z-\tau\bar{a}-a,$ which is equivalent to the claimed result.
\end{proof}

\begin{lemma}
\label{lem:GAFconditioned}If $G^{0}$ is a GAF conditioned to be zero at 0,
then $T_{a}G^{0}$ is a GAF conditioned to be zero at $a.$
\end{lemma}

One may prove this lemma using a orthogonal decomposition of the Gaussian measure on the Segal-Bargmann space, but instead we shall present a different route via direct calculation of the covariance structure.

\begin{proof}
Let $G^a$ be a GAF conditioned on $G(a)=0$. Since both $G^a$ and $T_{a}G^{0}$ are centered Gaussian processes, it suffices to compare their covariances.

The GAF conditioned to be zero at $0$, that is $G^0(z)=\sum_{k=1}^\infty \xi_k\frac{z^k}{k!}$, has covariance $\mathbb E [G^0(z)\overline{G^0(w)}]=e^{z\bar w}-1$. Hence, $ (T_{a}G^0)(z)= e^{-\left\vert a\right\vert ^{2}/2}e^{\bar{a}z}G^0(z-a)$ has covariance
$$\mathbb E\big[(T_{a}G^0)(z),\overline{(T_{a}G^0)(w)}\big]=e^{-|a|^2+\bar az +a\bar w}\big(e^{(z-a)(\overline{ w- a})}-1)=e^{z\bar w}-e^{-|a|^2+\bar az +a\bar w}.$$

On the other hand, the covariance function of a conditional Gaussian is well known (see for instance \cite[Exercise 2.1.3]{hough_etal_book}), which in our case expresses as
$$\mathbb E\big[ G^a(z)\overline{G^a(w)}\big]=\mathbb E\big[ G(z)\overline{G(w)}|G(a)=0\big]=e^{z\bar w} -e^{\bar a z}e^{-|a|^2}e^{a\bar w}.$$
Thus, both covariance functions coincide and the claim follows.
\end{proof}

\begin{proof}
[Proof of Theorem \ref{theo:distributionZeros}]Let $G^{0}$ be a GAF
conditioned to be zero at 0. Then by Lemma \ref{lem:GAFconditioned},
$T_{a}G^{0}$ is a GAF conditioned to be zero at $a.$ The result then follows
from Lemma \ref{lem:shiftZeros}.
\end{proof}

\section{Connection to the group $SL(2;\mathbb{R})$\label{meta.sec}}

\subsection{The metaplectic representation and its connection to the heat
flow}

In this section, we show that the operators $V_{\tau}$ in (\ref{eq:Vtau}) are
part of a family of operators $V(A),$ $A\in SL(2;\mathbb{R}),$ each of which
is a unitary operator on the Segal--Bargmann space in (\ref{eq:SBdef}), and
each of which preserves the GAF in distribution. We first note a convenient
representation of the operator $V_{\tau}$ as an integral operator.

\begin{proposition}
We can compute $V_{\tau}$ as an integral operator as follows:%
\begin{equation}
\lbrack V_{\tau}f](z)=\left(  1-\left\vert \tau\right\vert ^{2}\right)
^{1/4}~\int_{\mathbb{C}}\exp\left\{  \frac{\bar\tau}{2}z^{2}-\frac{\tau}{2}\bar{w}^{2}%
+\sqrt{1-\left\vert \tau\right\vert ^{2}}z\bar{w}\right\}  f(w)\frac
{e^{-\left\vert w\right\vert ^{2}}}{\pi}~dw. \label{VtauIntOp}%
\end{equation}

\end{proposition}

\begin{proof}
The result follows easily from the definition of $V_{\tau}$ in (\ref{VtauDef}%
), together with the integral representation of the heat operator in Theorem
\ref{theo:heatIntOp}.
\end{proof}

We have already asserted that $V_{\tau}$ is a unitary operator on the
Segal--Bargmann space. (Recall Definition \ref{def:SBspace}.) Actually, the
operators $V_{\tau},$ $\left\vert \tau\right\vert <1,$ form the main part of
collection of unitary operators on the Segal--Bargmann space coming from the
so-called metaplectic representation of the group $SL(2;\mathbb{R}),$ as we
now explain briefly. (See the appendix for more information.) In particular,
we will see that the composition of two operators of the form $V_{\tau}$ will
be another operator of the same form, up to a rotation and multiplication by a constant.

It will be convenient to use the \textquotedblleft complex variables
representation\textquotedblright\ of a matrix $A$ in $SL(2;\mathbb{R}),$ in
which we think of the associated linear map from $\mathbb{R}^{2}$ to
$\mathbb{R}^{2}$ as%
\[
z\mapsto pz+q\bar{z};\quad\bar{z}\mapsto\bar{q}z+\bar{p}\bar{z}%
\]
for some complex numbers $p$ and $q.$ Explicitly, we associate to $A\in
SL(2;\mathbb{R)}$ the matrix $A_{\mathbb{C}}$ given by
\begin{equation}
A_{\mathbb{C}}=\left(
\begin{array}
[c]{rr}%
p & q\\
\bar{q} & \bar{p}%
\end{array}
\right)  , \label{pqForm}%
\end{equation}
which is related to the usual matrix form of $A$ by%
\begin{equation}
\left(
\begin{array}
[c]{rr}%
p & q\\
\bar{q} & \bar{p}%
\end{array}
\right)  =\left(
\begin{array}
[c]{rr}%
1 & i\\
1 & -i
\end{array}
\right)  \left(
\begin{array}
[c]{rr}%
a & b\\
c & d
\end{array}
\right)  \left(
\begin{array}
[c]{rr}%
1 & i\\
1 & -i
\end{array}
\right)  ^{-1}, \label{conjugateSL2}%
\end{equation}
so that
\begin{equation}
p=\frac{a-ib+ic+d}{2};\quad q=\frac{a+ib+ic-d}{2}. \label{pAndQ}%
\end{equation}
A matrix of the form (\ref{pqForm}) comes from a matrix $A$ in
$SL(2;\mathbb{R})$ if and only if the matrix in (\ref{pqForm}) has determinant
1, that is, if and only if%
\begin{equation}
\left\vert p\right\vert ^{2}-\left\vert q\right\vert ^{2}=1. \label{p2q2}%
\end{equation}
The group of determinant-one matrices of the form (\ref{pqForm}) is the group
$SU(1;1).$

We follow the book \cite{folland_book} of G. Folland, Chapter 4, especially
Theorem 4.37, in the $n=1$ case. We associate to each matrix $A\in
SL(2;\mathbb{R})$ a \textit{pair} of integral operators $V(A)$ on the
Segal--Bargmann space given by
\begin{equation}
\lbrack V(A)f](z)=\pm\frac{1}{\sqrt{p}}\int_{\mathbb{C}}\exp\left\{  \frac
{1}{2}\frac{\bar q}{p}z^{2}-\frac{1}{2}\frac{q}{p}\bar{w}^{2}+\frac{1}{p}z\bar{w}\right\}
f(w)\frac{e^{-\left\vert w\right\vert ^{2}}}{\pi}~dw, \label{metaInt}%
\end{equation}
where $p$ and $q$ are as in (\ref{pAndQ}). Here we write $\pm$ to emphasize
that there is no preferred choice of the square root of the complex number
$p$---which means that we simply allow \textit{both} choices of the square root.

\begin{proposition}
For all $A,B\in SL(2;\mathbb{R})$ we have%
\begin{equation}
V(A)V(B)=\pm V(AB). \label{VAmult}%
\end{equation}
The operators $\{V_{A}\}_{A\in SL(2;\mathbb{R})}$ form a projective unitary
representation of $SL(2;\mathbb{R}),$ which can be made into an ordinary
representation of the connected double cover of $SL(2;\mathbb{R}).$
\end{proposition}

We now show how that every operator $V_{\tau}$ as in (\ref{eq:Vtau}) can be
expressed in the form $V_{\tau}=V(A),$ where $A$ is a positive symmetric
element of $SL(2;\mathbb{R}).$

\begin{theorem}
\label{theo:VAequalsVtau}Consider $A\in SL(2;\mathbb{R})$ with $p$ and $q$
defined by (\ref{pAndQ}). Then $A$ is symmetric and positive definite if and
only if $p$ is real and positive. In this case, if we set%
\begin{equation}
\tau=\frac{q}{p}, \label{tauDef}%
\end{equation}
and take the plus sign in (\ref{metaInt}), the operator $V_{\tau}$ in
(\ref{eq:Vtau}) will satisfy%
\[
V_{\tau}=V(A).
\]
Furthermore, for each $\tau\in\mathbb{D},$ there is a unique positive
symmetric matrix $A^{\tau}$ for which $q/p=\tau$ and this matrix satisfies%
\begin{equation}
A^{\tau}=\frac{1}{\sqrt{1-\left\vert \tau\right\vert ^{2}}}\left(
\begin{array}
[c]{cc}%
1+\operatorname{Re}\tau & \operatorname{Im}\tau\\
\operatorname{Im}\tau & 1-\operatorname{Re}\tau
\end{array}
\right)  ;\quad A_{\mathbb{C}}^{\tau}=\frac{1}{\sqrt{1-\left\vert
\tau\right\vert ^{2}}}\left(
\begin{array}
[c]{rr}%
1 & \tau\\
\bar{\tau} & 1
\end{array}
\right)  . \label{AcTau}%
\end{equation}

\end{theorem}

\begin{proof}
Note that from (\ref{pAndQ}), $p$ is real if and only if $b=c,$ that is, if
and only if $A$ is symmetric. If $p$ is real, then $p$ is positive if and only
if the trace of the matrix in (\ref{pqForm}) is positive, or, equivalently, if
and only if the trace of $A$ is positive. Thus, $p$ is real and positive if
and only if $A$ is symmetric and has positive trace. Since also $A$ is
$2\times2$ with $\det A=1,$ this condition is equivalent to $A$ being
symmetric and positive definite.

If $p>0,$ then from (\ref{p2q2}) we get $p=\sqrt{1+\left\vert q\right\vert
^{2}}.$ Then if $\tau$ is as in (\ref{tauDef}), we compute that
\[
1-\left\vert \tau\right\vert ^{2}=1-\frac{\left\vert q\right\vert ^{2}%
}{\left\vert p\right\vert ^{2}}=1-\frac{\left\vert q\right\vert ^{2}%
}{1+\left\vert q\right\vert ^{2}}=\frac{1}{1+\left\vert q\right\vert ^{2}},
\]
so that%
\begin{equation}
\frac{1}{p}=\sqrt{1-\left\vert \tau\right\vert ^{2}}. \label{1OverP}%
\end{equation}
Then we see that the operator in (\ref{metaInt}) matches the one in
(\ref{VtauIntOp}).
\end{proof}

We now look at various cases of the operators $V(A)$ for $A\in SL(2;\mathbb{R}%
).$

\begin{proposition}
\label{VAprops.prop}We have the following results.

\begin{enumerate}
\item The case where $A$ is a rotation corresponds to the case where the
matrix $A_{\mathbb{C}}$ in (\ref{pqForm}) is diagonal; specifically%
\begin{equation}
A=\left(
\begin{array}
[c]{rr}%
\cos\theta & -\sin\theta\\
\sin\theta & \cos\theta
\end{array}
\right)  \Longleftrightarrow A_{\mathbb{C}}=\left(
\begin{array}
[c]{rr}%
e^{i\theta} & 0\\
0 & e^{-i\theta}%
\end{array}
\right)  . \label{Arot}%
\end{equation}
In this case, the associated unitary operator is given by%
\begin{equation}
\lbrack V(A)f](z)=e^{-i\theta/2}f(e^{-i\theta}z). \label{Vrot}%
\end{equation}

\item \label{VArot.point}The case where $A$ is positive and diagonal
corresponds to the case where $p$ is real and positive and $q$ is real;
specifically%
\begin{equation}
A=\left(
\begin{array}
[c]{rr}%
e^{s} & 0\\
0 & e^{-s}%
\end{array}
\right)  \Longleftrightarrow A_{\mathbb{C}}=\left(
\begin{array}
[c]{rr}%
\cosh s & \sinh s\\
\sinh s & \cosh s
\end{array}
\right)  . \label{Adiag}%
\end{equation}
In this case, $\tau=\tanh s$ is real and negative and $V(A)=V_{\tau}$ by
Theorem \ref{theo:VAequalsVtau}.

\item \label{vaGen.point}For a general $A\in SL(2;\mathbb{R}),$ if we continue
to define $\tau$ by (\ref{tauDef}) and we take $\theta=\arg p,$ then%
\begin{equation}
\lbrack V(A)f](z)=\pm e^{-i\theta/2}[V_{\tau}f](e^{-i\theta}z). \label{VaVtau}%
\end{equation}
In particular, if $G$ is a GAF, then $V(A)G$ is again a GAF.
\end{enumerate}
\end{proposition}

\begin{proof}
The relation (\ref{Arot}) is just a computation, and the form of the operator
$V(A)$ in (\ref{Vrot}) follows from the definition (\ref{metaInt}) of the
metaplectic operators, together with the reproducing kernel identity
(\ref{reproKernel}). Similarly, (\ref{Adiag}) is just a computation and the
relation $V(A)=V_{\tau}$ follows from Theorem \ref{theo:VAequalsVtau}, using
(\ref{tauDef}).

For a general $A\in SL(2;\mathbb{R})$, we factor the matrix $A_{\mathbb{C}}$
in (\ref{pqForm}) as%
\begin{equation}
\left(
\begin{array}
[c]{rr}%
p & q\\
\bar{q} & \bar{p}%
\end{array}
\right)  =\left(
\begin{array}
[c]{cc}%
e^{i\theta} & 0\\
0 & e^{-i\theta}%
\end{array}
\right)  \left(
\begin{array}
[c]{cc}%
\left\vert p\right\vert  & qe^{-i\theta}\\
\bar{q}e^{i\theta} & \left\vert p\right\vert
\end{array}
\right)  , \label{pqFact}%
\end{equation}
with $\theta=\arg p.$ The second matrix on the right-hand side of
(\ref{pqFact}) has positive entries on the diagonal, so it has the form
$A_{\mathbb{C}}^{\tau}$ in (\ref{AcTau}), where%
\[
\tau=\frac{qe^{-i\theta}}{\left\vert p\right\vert }=\frac{q}{p},
\]
since $\theta=\arg p.$ The identity (\ref{VaVtau}) then follows from Point
\ref{VArot.point} and (\ref{VAmult}).
\end{proof}

\subsection{Hyperbolic invariance}

\label{sec:hyperbolic_invariance} Applying the heat flow to a GAF we obtain a
family of (essentially) GAF's indexed by the \textquotedblleft time
parameter\textquotedblright\ $\tau\in\mathbb{D}$, where the unit disk
$\mathbb{D}$ may be viewed as the Poincar\'{e}\ model of hyperbolic geometry.
It turns out that this family of GAF's enjoys certain invariance properties
with respect to the hyperbolic isometries of $\mathbb{D}$. To state the
corresponding result, we define a stochastic process $(Q_{\tau}(z))_{\tau
\in\mathbb{D},z\in\mathbb{C}}$ by
\begin{equation}
Q_{\tau}(z):=(V_{\tau}G)(z), \label{eq:Q_def}%
\end{equation}
where $V_{\tau}$ is as in (\ref{eq:Vtau}). Then, it follows from
Theorem~\ref{theo:GAF_invariant} that for every fixed $\tau\in\mathbb{D}$ the
random entire function $(Q_{\tau}(z))_{z\in\mathbb{C}}$ has the same
distribution as $(G(z))_{z\in\mathbb{C}}$. Recall that the group $SU(1,1)$
acts on $\mathbb{D}$ by the fractional-linear transformations of the form
\begin{equation}
\varphi(\tau)\equiv\varphi_{p,q}(\tau):=\frac{p\tau+q}{\bar{q}\tau+\bar{p}%
},\qquad p,q\in\mathbb{C},\;|p|^{2}-|q|^{2}=1. \label{eq:hyperbolic_isometry}%
\end{equation}

\begin{theorem}
\label{theo:hyperbolic_invariance} For every hyperbolic isometry
$\varphi:\mathbb{D}\rightarrow\mathbb{D}$ as in~\eqref{eq:hyperbolic_isometry}
the following equality of laws of stochastic processes holds:
\begin{equation}
\left(  \sqrt{\psi(\tau)}\cdot Q_{\varphi(\tau)}(\psi(\tau)z)\right)
_{\tau\in\mathbb{D},z\in\mathbb{C}}\overset{d}{=}(Q_{\tau}(z))_{\tau
\in\mathbb{D},z\in\mathbb{C}}, \label{HyperbolicInv}%
\end{equation}
where $\psi(\tau)$ is defined by
\[
\psi(\tau)\equiv\psi_{p,q}(\tau):=\frac{q\bar{\tau}+p}{|q\bar{\tau}+p|}%
\in\mathbb{T}.
\]
In (\ref{HyperbolicInv}), there is a unique continuous choice of the square
root of $\psi(\tau)$ once the square root at $\tau=0$ has been chosen.
\end{theorem}

%\ZK{I don't like that the factor $\psi(t)$ appears and that the function $Q_t$ depends on $t$ in a non-analytic way. At a first sight, it seems that it should be possible to re-define the function $Q_t$ such that it depends on $t$ analytically and the factor $\psi(t)$ disappears, but despite some attempts I was not able to do it. }

Let us now pass to zero sets in Theorem~\ref{theo:hyperbolic_invariance}.

\begin{corollary}
Let
$
Z_{\tau}=\mathcal{Z} (  \mathrm{e}^{-\tau D^{2}/2}G)  ,
$
where $\mathcal{Z}(\cdot)$ denotes the zero set of a function. Then for all
$\varphi\in SU(1,1),$ we have the following equality in distribution%
\[
\left(\psi(\tau)^{-1}\frac{Z_{\varphi(\tau)}}{\sqrt{1-\left\vert \varphi
(\tau)\right\vert ^{2}}}\right)_{\tau\in \mathbb{D}} 
\overset{d}{=}
\left(\frac{Z_{\tau}}{\sqrt{1-\left\vert
\tau\right\vert ^{2}}}\right)_{\tau\in \mathbb{D}}.
\]

\end{corollary}

We give two proofs of the theorem.

\begin{proof}
[First proof of Theorem \ref{theo:hyperbolic_invariance}]The first proof will
consist of fixing $\varphi$ as in (\ref{eq:hyperbolic_isometry}) and then
replacing the GAF $G$ in the definition (\ref{eq:Q_def}) by the function%
\[
V(A)G,
\]
where $V(\cdot)$ is the metaplectic representation defined by (\ref{metaInt})
and where $A$ is the matrix such that%
\[
A_{\mathbb{C}}=\left(
\begin{array}
[c]{rr}%
\bar{p} & q\\
\bar{q} & p
\end{array}
\right)  .
\]

Then $V(A)G$ is again a GAF (Point \ref{vaGen.point} of Proposition
\ref{VAprops.prop}), so that $(V_{\tau}V(A)G)(z)$ will have the same
distribution (as a function of $\tau$ and $z$) as $(V_{\tau}G)(z).$ We will
then compute that $(V_{\tau}V(A)G)(z)$ is equal to the left-hand side of
(\ref{HyperbolicInv}).

We recall from Theorem \ref{theo:VAequalsVtau} that $V_{\tau}$ is equal to
$V(A^{\tau}),$ where $A^{\tau}$ is as in (\ref{AcTau}). We then factor
$A_{\tau}^{\mathbb{C}}A_{\mathbb{C}}^{-1}$ as the product of two matrices, the
first being diagonal with diagonal entries having absolute value 1 and the
second having diagonal entries that are real and positive. (This factorization
corresponds to writing $AA_{\tau}$ as the product of a rotation and a positive
symmetric matrix.) A computation shows that the factorization is%
\begin{align*}
&  \left(
\begin{array}
[c]{rr}%
\frac{1}{\sqrt{1-\left\vert \tau\right\vert ^{2}}} & \frac{\tau}%
{\sqrt{1-\left\vert \tau\right\vert ^{2}}}\\
\frac{\bar{\tau}}{\sqrt{1-\left\vert \tau\right\vert ^{2}}} & \frac{1}%
{\sqrt{1-\left\vert \tau\right\vert ^{2}}}%
\end{array}
\right)  \left(
\begin{array}
[c]{rr}%
\bar{p} & q\\
\bar{q} & p
\end{array}
\right) \\
&  =\left(
\begin{array}
[c]{cc}%
1/\psi(\tau) & 0\\
0 & \psi(\tau)
\end{array}
\right)  \left(
\begin{array}
[c]{rr}%
\frac{1}{\sqrt{1-\left\vert \varphi(\tau)\right\vert ^{2}}} & \frac
{\varphi(\tau)}{\sqrt{1-\left\vert \varphi(\tau)\right\vert ^{2}}}\\
\frac{\overline{\varphi(\tau)}}{\sqrt{1-\left\vert \varphi(\tau)\right\vert
^{2}}} & \frac{1}{\sqrt{1-\left\vert \varphi(\tau)\right\vert ^{2}}}%
\end{array}
\right)  .
\end{align*}

Then, using (\ref{VAmult}), and Point \ref{VArot.point} of Proposition
\ref{VAprops.prop}, we find that
\begin{align*}
(V_{\tau}V(A)G)(z)  &  =\sqrt{\psi(\tau)}(V_{\varphi(\tau)}G)(\psi
(\tau)z)\\
&  =\sqrt{\psi(\tau)}\cdot Q_{\varphi(\tau)}(\psi(\tau)z).
\end{align*}
Since $V(A)G$ is again a GAF, the claimed result follows.
\end{proof}

\begin{proof}
[Second proof of Theorem \ref{theo:hyperbolic_invariance}]The second proof is
by direct computation of the covariances. Since we are dealing with
multivariate complex Gaussian processes, it suffices to check the equality of
covariance functions. Take some $z,w\in\mathbb{C}$ and $\tau,\sigma
\in\mathbb{D}$. Then, by~\eqref{eq:Q_def} and~\eqref{eq:cov_G_t_aux},
\begin{multline}
\mathbb{E}\left[  Q_{\tau}(z)\overline{Q_{\sigma}(w)}\right]  =\left(
\frac{(1-|\tau|^{2})^{1/2}(1-|\sigma|^{2})^{1/2}}{1-\tau\bar{\sigma}}\right)
^{1/2}\cdot\label{eq:covariance_Q}\\
\exp\left\{  z\overline{w}\cdot\frac{(1-|\tau|^{2})^{1/2}(1-|\sigma
|^{2})^{1/2}}{1-\tau\bar{\sigma}}\right\}  \cdot\exp\left\{  \frac{1}{2}%
z^{2}\cdot\frac{\bar{\tau}-\bar{\sigma}}{1-\tau\bar{\sigma}}+\frac{1}{2}%
\bar{w}^{2}\cdot\frac{\sigma-\tau}{1-\tau\bar{\sigma}}\right\}  .
\end{multline}
It is easy to check that
\[
1-\varphi(\tau)\overline{\varphi(\sigma)}=\frac{1-\tau\bar{\sigma}}{(\bar
{q}\tau+\bar{p})(q\bar{\sigma}+p)}\quad\text{ and }\quad\varphi(\tau
)-\varphi(\sigma)=\frac{\tau-\sigma}{(\bar{q}\tau+\bar{p})(\bar{q}\sigma
+\bar{p})}.
\]
In particular, as a special case of the first identity we obtain
\[
1-|\varphi(\tau)|^{2}=\frac{1-|\tau|^{2}}{|q\bar{\tau}+p|^{2}},\qquad
1-|\varphi(\sigma)|^{2}=\frac{1-|\sigma|^{2}}{|q\bar{\sigma}+p|^{2}}.
\]
It follows that
\begin{align*}
\frac{(1-|\varphi(\tau)|^{2})^{1/2}(1-|\varphi(\sigma)|^{2})^{1/2}}%
{1-\varphi(\tau)\overline{\varphi(\sigma)}}  &  =\frac{(1-|\tau|^{2}%
)^{1/2}(1-|\sigma|^{2})^{1/2}}{1-\tau\bar{\sigma}}\cdot\frac{\psi(\sigma
)}{\psi(\tau)},\\
\frac{\overline{\varphi(\tau)}-\overline{\varphi(\sigma)}}{1-\varphi
(\tau)\overline{\varphi(\sigma)}}  &  =\frac{\bar{\tau}-\bar{\sigma}}%
{1-\tau\bar{\sigma}}\cdot\frac{\bar{q}\tau+\bar{p}}{q\bar{\tau}+p}=\frac
{\bar{\tau}-\bar{\sigma}}{1-\tau\bar{\sigma}}\cdot\frac{1}{\psi^{2}(\tau)},\\
\frac{\varphi(\sigma)-\varphi(\tau)}{1-\varphi(\tau)\overline{\varphi(\sigma
)}}  &  =\frac{\sigma-\tau}{1-\tau\bar{\sigma}}\cdot\frac{q\bar{\sigma}%
+p}{\bar{q}\sigma+\bar{p}}=\frac{\sigma-\tau}{1-\tau\bar{\sigma}}\cdot\psi
^{2}(\sigma).
\end{align*}
Using all these identities together with~\eqref{eq:covariance_Q}, we arrive
at
\[
\mathbb{E}\left[  Q_{\varphi(\tau)}(\psi(\tau)z)\overline{Q_{\varphi(\sigma
)}(\psi(\sigma)w)}\right]  =\left(  \frac{\psi(\sigma)}{\psi(\tau)}\right)
^{1/2}\mathbb{E}\left[  Q_{\tau}(z)\overline{Q_{\sigma}(w)}\right]  ,
\]
which proves the claim.
\end{proof}

\section{Differential equations for the zeros}

In this section, we consider systems of differential equations for the zeros
of a entire functions of different orders evolving according to the heat flow.
In the case of a function of order $\rho<1,$ we recover earlier results of
Papanicolaou, Kallitsi, and Smyrlis \cite{PKS}, but we also obtain results for
functions of order $\rho<2$ and functions of order $\rho=2$ and finite type.
We then construct an iterative method for solving the system in each case.

\subsection{The first and second derivatives\label{FirstSecond.sec}}

In this section, it is convenient to introduce a sub-case of Situation (S1),
which we call (S0), namely when $F$ is an entire holomorphic function of order
$\rho<1.$ We begin with the (S0) case, where the formula is the same as in the
polynomial case. In this (S0) case, the result was obtained previously by
Papanicolaou, Kallitsi, and Smyrlis \cite[Equation (4.40)]{PKS}.

\begin{theorem}
[First derivatives for (S0) case]\label{S0deriv.thm}Suppose $F$ is an entire
function of order $\rho<1$ and define $F(z,\tau)$ by%
\[
F(z,\tau)=(e^{-\tau D^{2}/2}F)(z),
\]
for all $\tau$ and $z$ in $\mathbb{C}.$ Fix $\tau_{0}\in\mathbb{C}$, assume
$z_{0}$ is a simple zero of $F(z,\tau_{0}),$ and let $z(\tau)$ be the unique
holomorphic function defined near $\tau_{0}$ such that $z(\tau_{0})=z_{0}$ and
$F(z(\tau),\tau)=0.$ Then we have the following formulas for the derivative of
$z(\tau)$ for $\tau$ near $\tau_{0}$:%
\begin{equation}
z^{\prime}(\tau)=\sum_{w\in\mathbb{C},~F(w,\tau)=0}\frac{\mathbf{1}_{\{w\neq
z_{0}\}}}{z(\tau)-w}, \label{zprime0}%
\end{equation}
where the zeros of $F(z,\tau)$ are listed with their multiplicities.
\end{theorem}

We then turn to the (S1) case.

\begin{theorem}
[First derivatives for (S1) case]\label{S1deriv.thm}Suppose $F$ is an entire
function of order $\rho<2$ and define $F(z,\tau)$ by%
\[
F(z,\tau)=(e^{-\tau D^{2}/2}F)(z),
\]
for all $\tau$ and $z$ in $\mathbb{C}.$ Fix $\tau_{0}$ and $c$ in
$\mathbb{C},$ assume $F(c,\tau_{0})\neq0,$ and define%
\begin{equation}
a_{1}(\tau)=\partial_{1}\log F(c,\tau) \label{a1def}%
\end{equation}
for $\tau$ near $\tau_{0}.$ Suppose also that $z_{0}$ is a simple zero of
$F(z,\tau_{0})$ and let $z(\tau)$ be the unique holomorphic function defined
near $\tau_{0}$ such that $z(\tau_{0})=z_{0}$ and $F(z(\tau),\tau)=0.$ Then we
have the following formulas for the derivatives of $a_{1}(\tau)$ and $z(\tau)$
for $\tau$ near $\tau_{0}$:%
\begin{align}
a_{1}^{\prime}(\tau)  &  =a_{1}(\tau)\sum_{w\in\mathbb{C},F(w,\tau)=0}\frac
{1}{(w-c)^{2}}+\sum_{w\in\mathbb{C},F(w,\tau)=0}\frac{1}{(w-c)^{3}%
}\label{a1Prime1}\\
z^{\prime}(\tau)  &  =a_{1}(\tau)+\sum_{w\in\mathbb{C},F(w,\tau)=0}\left(
\frac{\mathbf{1}_{\{w\neq z(\tau)\}}}{z(\tau)-w}+\frac{1}{w-c}\right)  ,
\label{zzPrime1}%
\end{align}
where the zeros of \thinspace$F(z,\tau)$ are listed with their multiplicities.
\end{theorem}

Finally, we turn to the (S2) case.

\begin{theorem}
[First derivatives for (S2) case]\label{S2deriv.thm}Suppose $F$ is an entire
function of order $\rho=2$ and finite type $\sigma,$ and define $F(z,\tau)$ by%
\[
F(z,\tau)=(e^{-\tau D^{2}/2}F)(z),
\]
for $\left\vert \tau\right\vert <1/(2\sigma)$ and $z$ in $\mathbb{C}.$ Fix
$\tau_{0}$ and $c$ in $\mathbb{C},$ assume $F(c,\tau_{0})\neq0,$ and define%
\begin{align*}
a_{1}(\tau)  &  =\partial_{1}\log F(c,\tau)\\
a_{2}(\tau)  &  =\partial_{1}^{2}\log F(c,\tau)
\end{align*}
for $\tau$ near $\tau_{0}.$ Suppose also that $z_{0}$ is a simple zero of
$F(z,\tau_{0})$ and let $z(\tau)$ be the unique holomorphic function defined
near $\tau_{0}$ such that $z(\tau_{0})=z_{0}$ and $F(z(\tau),\tau)=0.$ Then we
have the following formulas for the derivatives of $a_{1}(\tau),$ $a_{2}%
(\tau),$ and $z(\tau)$ for $\tau$ near $\tau_{0}$:%
\begin{align}
a_{1}^{\prime}(\tau)  &  =-a_{1}(\tau)a_{2}(\tau)+\sum_{w\in\mathbb{C}%
,F(w,\tau)=0}\frac{1}{(w-c)^{3}}\label{a1Prime2}\\
a_{2}^{\prime}(\tau)  &  =-a_{2}(\tau)^{2}+2a_{1}(\tau)\sum_{w\in
\mathbb{C},F(w,\tau)=0}\frac{1}{(w-c)^{3}}+3\sum_{w\in\mathbb{C},F(w,\tau
)=0}\frac{1}{(w-c)^{4}}\label{a2Prime2}\\
z^{\prime}(\tau)  &  =a_{1}(\tau)+a_{2}(\tau)(z(\tau)-c)\nonumber\\
&  +\sum_{w\in\mathbb{C},F(w,\tau)=0}\left(  \frac{\mathbf{1}_{\{w\neq
z(\tau)\}}}{z(\tau)-w}+\frac{1}{w-c}+\frac{z(\tau)-c}{(w-c)^{2}}\right)  ,
\label{zzPrime2}%
\end{align}
where the zeros of \thinspace$F(z,\tau)$ are listed with their multiplicities.
\end{theorem}

In the previous theorems, we do not assume that \textit{all} the zeros of
$F(z,\tau)$ are simple, but only that $z_{0}$ is a simple zero of
$F(z,\tau_{0}).$ In particular, there is no ambiguity in any of the formulas
if the same $w\in\mathbb{C}$ occurs more than once.

We now turn to the computation of the second derivatives of the zeros with
respect to $\tau,$ and find that we obtain the same formula in all three cases
S0, S1, and S2, namely the same rational Calogero--Moser equation we have in
the polynomial case.

\begin{theorem}
[Second derivatives]\label{secDeriv.thm}Suppose $F$ is an entire function of
order $\rho<2$ or order $\rho=2$ and finite type $\sigma,$ and define
$F(z,\tau)$ by%
\[
F(z,\tau)=(e^{-\tau D^{2}/2}F)(z),
\]
for all $z\in\mathbb{C},$ where we assume $\left\vert \tau\right\vert
<1/(2\sigma)$ if $\rho=2.$ Fix $\tau_{0}\in\mathbb{C}$, assume $z_{0}$ is a
simple zero of $F(z,\tau_{0}),$ and let $z(\tau)$ be the unique holomorphic
function defined near $\tau_{0}$ such that $z(\tau_{0})=z_{0}$ and
$F(z(\tau),\tau)=0.$ Then we have the following formula for the second
derivative of $z(\tau)$ for $\tau$ near $\tau_{0}$:%
\[
z^{\prime\prime}(\tau)=-2\sum_{w\in\mathbb{C},F(w,\tau)=0}\frac{\mathbf{1}%
_{\{w\neq z_{0}\}}}{(z(\tau)-w)^{3}}.
\]

\end{theorem}

\begin{proof}
[Proof of Theorems \ref{S0deriv.thm}, \ref{S1deriv.thm}, and \ref{S2deriv.thm}%
]We use the subscript notation for partial derivatives and we write $F=e^{H},$
where $H=\log F,$ away from the zeros of $F.$ To compute $z^{\prime}(\tau),$
we will use Proposition \ref{diffZero.prop} and compute that
\begin{equation}
z^{\prime}(\tau)=\lim_{z\rightarrow z_{0}}\left(  \frac{F_{z}(z,\tau
)}{F(z,\tau)}-\frac{1}{z-z(\tau)}\right)  =\lim_{z\rightarrow z_{0}}\left(
H_{z}(z,\tau)-\frac{1}{z-z(\tau)}\right)  . \label{zprimeProof}%
\end{equation}

We then recall the Hadamard factorizations of $F(z,\tau)$ in the cases S0, S1,
and S2, respectively, writing the zeros (for this one fixed $\tau$) as
$\{z_{j}\}_{j}$, with the distinguished zero $z(\tau)$ corresponding to the
case $j=0$. To make the notation more compact, we assume $c=0$, giving%
\begin{align}
F(z,\tau)  &  =e^{a_{0}}\prod_{k}\left(  1-\frac{z}{z_{k}}\right)
\label{Hadamard1}\\
F(z,\tau)  &  =e^{a_{0}+a_{1}z}\prod_{k}\left(  1-\frac{z}{z_{k}}\right)
e^{z/z_{k}}\label{Hadamard2}\\
F(z,\tau)  &  =e^{a_{0}+a_{1}z+\frac{1}{2}a_{2}z^{2}}\prod_{k}\left(
1-\frac{z}{z_{k}}\right)  e^{z/z_{k}+\frac{1}{2}z^{2}/z_{k}^{2}}.
\label{Hadamard3}%
\end{align}

Then we compute $H(z,\tau)=\log F(z,\tau)$ in the three cases, separating out
the zero $z(\tau)$ from the other zeros:%
\begin{align*}
H(z,\tau)  &  =a_{0}+\log\left(  1-\frac{z}{z(\tau)}\right)  +\sum_{k\neq
0}\log\left(  1-\frac{z}{z_{k}}\right) \\
H(z,\tau)  &  =a_{0}+a_{1}z+\log\left(  1-\frac{z}{z(\tau)}\right)  +\frac
{z}{z_{0}}\\
&  +\sum_{k\neq0}\left(  \log\left(  1-\frac{z}{z_{k}}\right)  +\frac{z}%
{z_{k}}\right) \\
H(z,\tau)  &  =a_{0}+a_{1}z+\frac{1}{2}a_{2}z^{2}+\log\left(  1-\frac
{z}{z(\tau)}\right)  +\frac{z}{z_{0}}+\frac{1}{2}\frac{z^{2}}{z_{0}^{2}}\\
&  +\sum_{k\neq0}\left(  \log\left(  1-\frac{z}{z_{k}}\right)  +\frac{z}%
{z_{k}}+\frac{1}{2}\frac{z^{2}}{z_{k}^{2}}\right)  .
\end{align*}
It is then a direct computation to compute (\ref{zprimeProof}) in each of the
three cases to obtain (\ref{zprime0}), (\ref{zzPrime1}), and (\ref{zzPrime2}).

Finally, we compute the derivatives of $a_{1}(\tau)$ in the (S1) case and
$a_{1}(\tau)$ and $a_{2}(\tau)$ in the (S2) case. For notational simplicity, we
take $c=0.$ Noting that
\[
\log\left(  1-\frac{z}{z_{k}}\right)  =-\frac{z}{z_{k}}-\frac{1}{2}\frac
{z^{2}}{z_{j}^{2}}-\frac{1}{3}\frac{z^{3}}{z_{j}^{3}}-\cdots,
\]
we find that $H(z,\tau)=\log F(z,\tau)$ is computed in the (S1) and (S2) cases,
respectively, as%
\begin{align}
H(z,\tau)  &  =a_{0}(\tau)+a_{1}(\tau)z+\sum_{k}\left(  -\frac{1}{2}%
\frac{z^{2}}{z_{k}(\tau)^{2}}-\frac{1}{3}\frac{z^{3}}{z_{k}(\tau)^{3}}%
+O(z^{4})\right) \label{LogH1}\\
H(z,\tau)  &  =a_{0}(\tau)+a_{1}(\tau)z+\frac{1}{2}a_{2}(\tau)z^{2}+\sum
_{k}\left(  -\frac{1}{3}\frac{z^{3}}{z_{k}(\tau)^{3}}-\frac{1}{4}\frac{z^{4}%
}{z_{k}(\tau)}+O(z^{5})\right)  . \label{LogH2}%
\end{align}

Then
\begin{equation}
\frac{\partial H}{\partial\tau}=-\frac{1}{2}\frac{F_{zz}}{F}=-\frac{1}%
{2}(H_{z}^{2}+H_{zz}). \label{calcFzzOverF}%
\end{equation}
Similarly,%
\begin{equation}
\frac{\partial}{\partial\tau}\frac{\partial H}{\partial z}=\frac{\partial
}{\partial z}\frac{\partial H}{\partial\tau}=-H_{z}H_{zz}-\frac{1}{2}H_{zzz},
\label{GzTau}%
\end{equation}
so that in the (S1) and (S2) cases, we have%
\begin{equation}
a_{1}^{\prime}(\tau)=\left.  \left(  -H_{z}H_{zz}-\frac{1}{2}H_{zzz}\right)
\right\vert _{z=0}. \label{da1Dtau}%
\end{equation}
We may now evaluate $a_{1}^{\prime}(\tau)$ (\ref{LogH1}) or (\ref{LogH2}) in
the (S1) and (S2) cases, respectively to obtain (\ref{a1Prime1}) and
(\ref{a1Prime2}).

Finally, we compute%
\begin{align*}
\frac{\partial}{\partial\tau}H_{zz}  &  =\frac{\partial}{\partial z}\left(
-H_{z}H_{zz}-\frac{1}{2}H_{zzz}\right) \\
&  =-H_{zz}^{2}-H_{z}H_{zzz}-\frac{1}{2}H_{zzzz}.
\end{align*}
Evaluating at $z=0$ gives%
\[
\frac{da_{2}}{d\tau}=\left.  \left(  -H_{zz}^{2}-H_{z}H_{zzz}-\frac{1}%
{2}H_{zzzz}\right)  \right\vert _{z=0},
\]
and we may easily compute this last expression from (\ref{LogH2}) to obtain
(\ref{a2Prime2}).
\end{proof}

\begin{proof}
[Proof of Theorem \ref{secDeriv.thm}]Recall from Proposition
\ref{diffZero.prop} that
\begin{equation}
z^{\prime}(\tau)=\frac{1}{2}\frac{F_{zz}(z(\tau),\tau)}{F_{z}(z(\tau),\tau)}.
\label{zPrimeRepeat}%
\end{equation}
If we then differentiate the identity $F(z(\tau),\tau)=0$ twice with respect
to $\tau$ and use (\ref{zPrimeRepeat}), we easily obtain the formula%
\begin{equation}
z^{\prime\prime}(\tau)=-\frac{1}{4}\left(  \frac{F_{zz}}{F_{z}}\right)
^{3}+\frac{1}{2}\frac{F_{zz}}{F_{z}}\frac{F_{zzz}}{F_{z}}-\frac{1}{4}%
\frac{F_{zzzz}}{F_{z}}, \label{zDoublePrime}%
\end{equation}
where the right-hand side is evaluated at $z=z(\tau).$

We now write%
\[
F(z(\tau),\tau)=\left(  1-\frac{z}{z(\tau)}\right)  F^{\mathrm{reg}}%
(z(\tau),\tau),
\]
where $F^{\mathrm{reg}}$ is the \textquotedblleft regular\textquotedblright%
\ part of $F,$ meaning that it is nonzero at $(z(\tau),\tau).$ Next, we verify
that for $m\geq1,$ we have%
\[
F^{(m)}(z,\tau)=-\frac{1}{z(\tau)}m(F^{\mathrm{reg}})^{(m-1)}(z,\tau)+\left(
1-\frac{z}{z(\tau)}\right)  (F^{\mathrm{reg}})^{(m)}(z,\tau),
\]
where $F^{(m)}$ is the $m$th derivative of $F$ with respect to $z.$ It follows
that%
\begin{equation}
\frac{F^{(m)}(z(\tau),\tau)}{F_{z}(z(\tau),\tau)}=m\frac{(F^{\mathrm{reg}%
})^{(m-1)}(z(\tau),\tau)}{F^{\mathrm{reg}}(z(\tau),\tau)}. \label{FmOverFz}%
\end{equation}
We then use (\ref{FmOverFz}) in each term of (\ref{zDoublePrime}), giving
\begin{equation}
z^{\prime\prime}(\tau)=-\left(  \frac{2(F_{z}^{\mathrm{reg}})^{3}%
-3F^{\mathrm{reg}}F_{z}^{\mathrm{reg}}F_{zz}^{\mathrm{reg}}+(F^{\mathrm{reg}%
})^{2}F_{zzz}^{\mathrm{reg}}}{(F^{\mathrm{reg}})^{3}}\right)  .
\label{zDouble2}%
\end{equation}
Finally, if we write $F^{\mathrm{reg}}=e^{H^{\mathrm{reg}}},$ a calculation
shows that (\ref{zDouble2}) simplifies to
\begin{equation}
z^{\prime\prime}(\tau)=-H_{zzz}^{\mathrm{reg}}. \label{zDoubleHreg}%
\end{equation}

We now use Hadamard factorization of $F(z(\tau),\tau)$ in each of the three
cases S0, S1, and S2. If we consider, say, the (S2) case and take the base point
$c$ to be 0 for notational simplicity, we find that
\begin{align}
H^{\mathrm{reg}}(z,\tau)  &  =a_{0}+a_{1}z+\frac{1}{2}a_{2}z^{2}\nonumber\\
&  +\sum_{w\in\mathbb{C},~F(z,\tau)=0}\left(  \mathbf{1}_{\{w\neq z(\tau
)\}}\log\left(  1-\frac{z}{w}\right)  +\frac{z}{w}+\frac{1}{2}\frac{z^{2}%
}{w^{2}}\right)  . \label{Hreg}%
\end{align}
When we take the third derivative of $H^{\mathrm{reg}}$ with respect to $z,$
most of the terms in (\ref{Hreg}) become zero and (\ref{zDoubleHreg}) reduces
to the claimed expression for $z^{\prime\prime}(\tau).$ The cases (S0) and S1
are entirely similar.
\end{proof}

We conclude this section by considering the rescaled zeros of the GAF, as in our main Theorem \ref{theo:GAF_invariant_generalization_zeros_introduction}. We consider heat-evolved GAF $G(z,t)$ for
\textit{real} $\tau$ between $(-1,1),$ and then rescale the zeros by dividing
by $\sqrt{1-\tau^{2}}.$ The corollary tells us that these rescaled zeros have
the same distribution as the zeros of the original GAF. We now obtain a
differential equation for the rescaled zeros.

\begin{proposition}
[Derivatives of the rescaled GAF zeros]For real values of $\tau$ between $-1$
and $1$ enumerate the zeros of the heat evolved GAF as $\{z_{j}(\tau)\}_{j}$
and let $\{y_{j}(\tau)\}_{j}$ be the rescaled zeros%
\[
y_{j}(\tau)=\frac{z_{j}(\tau)}{\sqrt{1-\tau^{2}}},\quad-1<\tau<1.
\]
Let $a_{1}(\tau)$ and $a_{2}(\tau)$ be as in Theorem \ref{S2deriv.thm} with
$c=0$ and introduce a new time variable%
\[
s=\tanh^{-1}t.
\]
Then if $z_{j}(\tau)$ is a simple zero of $G(z,\tau),$ we have the following
formula for the $s$-derivative of the rescaled zero $y_{j}(s):$
\begin{align}
\frac{dy_{j}}{ds}  &  =\frac{a_{1}(s)}{\cosh s}+y_{j}(s)\left(  \frac
{a_{2}(s)}{\cosh^{2}s}+\tanh s\right) \nonumber\\
&  +\sum_{k}\left(  \frac{\mathbf{1}_{\{k\neq j\}}}{y_{j}(s)-y_{k}(s)}%
+\frac{1}{y_{k}(s)}+\frac{y_{j}(s)}{y_{k}(s)^{2}}\right)  . \label{dyds}%
\end{align}

\end{proposition}

Note that the sum over $k$ on the right-hand side of (\ref{dyds}) has exactly
the same form as the sum in (\ref{zzPrime2}) (with $c=0$).

\begin{proof}
We first observe that, in terms of the new time variable $s,$ we have%
\[
\frac{1}{\sqrt{1-\tau^{2}}}=\cosh s,
\]
so that%
\[
y_{j}(s)=z_{j}(\tanh s)\cosh s.
\]
Then we compute that
\begin{align}
y_{j}^{\prime}(s)  &  =\frac{z_{j}^{\prime}(\tanh s)}{\cosh s}+z_{j}(\tanh
s)\sinh s\nonumber\\
&  =\frac{z_{j}^{\prime}(\tanh s)}{\cosh s}+y_{j}(\tanh s)\tanh s.
\label{yjPrime.eq}%
\end{align}

We then use (\ref{zzPrime2}) (with $c=0$) to compute the first term on the
right-hand side of (\ref{yjPrime.eq}):
\begin{align*}
\frac{z_{j}^{\prime}(\tanh s)}{\cosh s}  &  =\frac{1}{\cosh s}\sum_{k}\left(
\frac{\mathbf{1}_{k\neq j}}{z_{j}(\tanh s)-z_{k}(\tanh s)}+\frac{1}%
{z_{k}(\tanh s)}+\frac{z_{j}(\tanh s)}{z_{k}(\tanh s)^{2}}\right) \\
&  +\frac{1}{\cosh s}a_{1}(\tanh s)+\frac{1}{\cosh s}a_{2}(\tanh s)z_{j}(\tanh
s)\\
&  =\sum_{k}\left(  \frac{\mathbf{1}_{k\neq j}}{y_{j}(s)-z_{k}(s)}+\frac
{1}{y_{k}(s)}+\frac{y_{j}(s)}{y_{k}(s)^{2}}\right) \\
&  +a_{1}(\tanh s)\operatorname{sech}s+y_{j}(s)a_{2}(\tanh
s)\operatorname{sech}^{2}s.
\end{align*}
Combining this result with the second term on the right-hand side of
(\ref{yjPrime.eq}) gives the claimed result.
\end{proof}

\subsection{Iterating the equations}

Let us consider, for definiteness, the (S2) case, and let us assume that
\textit{all} the zeros of $F(z,\tau)$ are simple. Then we can enumerate them
as $\{z_{j}(\tau)\}_{j}$ and apply Theorem \ref{S2deriv.thm} to every zero,
giving
\begin{align}
a_{1}^{\prime}(\tau)  &  =-a_{1}(\tau)a_{2}(\tau)+\sum_{k}\frac{1}{(z_{k}%
(\tau)-c)^{3}}\label{ODE1}\\
a_{2}^{\prime}(\tau)  &  =-a_{2}(\tau)^{2}+2a_{1}(\tau)\sum_{k}\frac{1}%
{(z_{k}(\tau)-c)^{3}}+3\sum_{k}\frac{1}{(z_{k}(\tau)-c)^{4}}\label{ODE2}\\
z_{j}^{\prime}(\tau)  &  =a_{1}(\tau)+a_{2}(\tau)(z(\tau)-c)\nonumber\\
&  +\sum_{k}\left(  \frac{\mathbf{1}_{\{k\neq j\}}}{z_{j}(\tau)-z_{k}(\tau
)}+\frac{1}{z_{k}(\tau)-c}+\frac{z_{j}(\tau)-c}{(z_{k}(\tau)-c)^{2}}\right)  .
\label{ODE3}%
\end{align}
We would then like to think of (\ref{ODE1}), (\ref{ODE2}), and (\ref{ODE3}) as
an infinite system of ODEs for the quantities $a_{1},$ $a_{2},$ and
$\{z_{j}\}_{j}.$ Indeed, we can consider this system apart from any connection
to the heat flow, and prove a uniqueness result as follows.

\begin{theorem}
\label{uniqueSoln.thm}Suppose $a_{1}(\cdot)$, $a_{2}(\cdot)$, and
$\{z_{j}(\cdot)\}_{j}$ are holomorphic functions defined on a disk $D$ in the
plane. Then we say that these functions constitute a solution to
(\ref{ODE1})--(\ref{ODE3}) on $D$ if the quantities $z_{j}(\tau)$ are distinct
for each $\tau\in D$ and these three equations hold for every $\tau\in D,$
with locally uniform convergence of the series on the right-hand side of each
equation. Suppose that $a_{1}(\cdot),$ $a_{2}(\cdot),$ and $\{z_{j}%
(\cdot)\}_{j}$ and $\hat{a}_{1}(\cdot),$ $\hat{a}_{2}(\cdot),$ and $\{\hat
{z}_{j}(\cdot)\}_{j}$ are two solutions to the system that agree at some
$\tau_{0}\in D.$ Then the two solutions agree everywhere on $D.$
\end{theorem}

\begin{proof}
A locally uniformly convergent series of analytic functions can be
differentiated any number of times. Differentiating each of (\ref{ODE1}%
)--(\ref{ODE3}) $m-1$ times and evaluating at $\tau=\tau_{0}$ expresses
$a_{1}^{(m)}(\tau_{0}),$ $a_{2}^{(m)}(\tau_{0}),$ and $z_{j}^{(m)}(\tau_{0})$
in terms of $a_{1}^{(l)}(\tau_{0}),$ $a_{2}^{(l)}(\tau_{0}),$ and $z_{j}%
^{(l)}(\tau_{0})$ for $0\leq l\leq m-1.$ Thus, inductively, all the
derivatives of $a_{1}(\cdot),$ $a_{2}(\cdot),$ and $\{z_{j}(\cdot)\}_{j}$ at
$\tau_{0}$ agree with the corresponding derivatives of $\hat{a}_{1}(\cdot),$
$\hat{a}_{2}(\cdot),$ and $\{\hat{z}_{j}(\cdot)\}_{j}$ at $\tau_{0}.$
\end{proof}

Similar theorems hold also in the (S0) and (S1) cases. There is, however, a
difficulty in applying these theorems---in any of the three cases---to
functions obtained from the heat flow, as explained in the following remark.

\begin{remark}
Suppose that $F(z,\tau)$ is as in Theorem \ref{S2deriv.thm}, that all the
zeros of $F(z,\tau_{0})$ are simple, and that $F(z,\tau_{0})$ has infinitely
many zeros. Then it may be that for every disk $D$ around $\tau_{0},$ there
exists $\tau\in D$ for which some zero of $F(z,\tau)$ is not simple. Thus, if
the functions $z_{j}(\cdot)$ are the zeros of $F(z,\tau)$ and $a_{1}(\cdot)$
and $a_{2}(\cdot)$ are as in Theorem \ref{S2deriv.thm}, there may be no single
disk $D$ on which each function is defined and holomorphic, in which case,
Theorem \ref{uniqueSoln.thm} will not apply.
\end{remark}

Despite the difficulty described in the remark, there is a sense in which we
can iterate the equations (\ref{ODE1})--(\ref{ODE3}), similarly to the proof
of Theorem \ref{uniqueSoln.thm} to obtain a uniqueness result. Now, we should
point out that if (in, say, the (S2) case), the quantities $a_{1},$ $a_{2}$, and
$\{z_{j}\}_{j=1}^{\infty}$ are known at time $\tau_{0},$ then the function
$F(z,\tau_{0})$ can be recovered up to a constant as a Hadamard product. Then
the we can apply the heat flow and uniquely determine the function $F(z,\tau)$
for any $\tau,$ so that $a_{1}(\tau),$ $a_{2}(\tau),$ and $\{z_{j}%
(\tau)\}_{j=1}^{\infty}$ are uniquely determined for all $\tau.$ But this
argument does not provide any mechanism for actually \textit{computing} the
zeros at time $\tau.$ By contrast, Theorem \ref{trunc.thm} and the
computations in Sections \ref{S0S1.sec} and \ref{S2.sec} give us
\textit{explicit} formulas for all the derivatives of each $z_{j}(\tau)$ at
$\tau=\tau_{0},$ in terms of the initial values of $\{z_{k}(\tau_{0}%
)\}_{k=1}^{\infty}$ (and possibly $a_{1}(\tau_{0})$ and $a_{2}(\tau_{0}),$
depending on the case).

\begin{theorem}
[Truncation and iteration theorem]\label{trunc.thm}Suppose $F$ is as in one of
Theorems \ref{S0deriv.thm}, \ref{S1deriv.thm}, or \ref{S2deriv.thm} and
suppose that, for some $\tau_{0},$ $F(z,\tau_{0})$ has infinitely many zeros,
each of which is simple, listed in some fixed order as $\{z_{j}(\tau
_{0})\}_{j=1}^{\infty}.$ Consider the appropriate Hadamard product
representation of $F(z,\tau_{0}),$ as in (\ref{Hadamard1}), (\ref{Hadamard2}),
or (\ref{Hadamard3}), respectively. Let $F^{N}(z,\tau_{0})$ be obtained by
truncating the product after $N$ factors, and let
\begin{equation}
F^{N}(z,\tau)=e^{-\frac{(\tau-\tau_{0})}{2}D^{2}}F^{N}(z,\tau_{0})
\label{FNdef}%
\end{equation}
for $\tau$ sufficiently close to $\tau_{0}.$ Finally, for $j\leq N,$ let
$z_{j}^{N}(\tau)$ denote the zero of $F^{N}(z,\tau)$ that equals $z_{j}%
^{N}(\tau_{0})$ at $\tau=\tau_{0}.$ Then for all $m\geq1,$ we have%
\[
z_{j}^{(m)}(\tau_{0})=\lim_{N\rightarrow\infty}(z_{j}^{N})^{(m)}(\tau_{0}).
\]
Furthermore, $(z_{j}^{N})^{(m)}(\tau_{0})$ may be computed by iterating the
appropriate system of ordinary differential equations, as in (\ref{ODE1}%
)--(\ref{ODE3}) in the (S2) case.
\end{theorem}

The theorem, together with the computations in Sections \ref{S0S1.sec} and
\ref{S2.sec}, will allow us to compute \textit{all} the derivatives of each
$z_{j}(\tau)$ at $\tau=\tau_{0}.$ Thus, each $z_{j}(\tau)$ can be recovered
uniquely from these formulas on its disk of convergence around $\tau_{0}.$
Note that although there may be no \textit{single} disk around $\tau_{0}$ on
which every $z_{j}(\tau)$ is defined, Proposition \ref{diffZero.prop}
guarantees that for each fixed $j,$ there is a disk of radius $r_{j}$ around
$\tau_{0}$ on which $z_{j}(\tau)$ is defined and holomorphic. (We just have no
reason to expect that the $r_{j}$'s are bounded away from zero as $j$ varies.)

We will examine the computation of the higher derivatives in detail in
Sections \ref{S0S1.sec} and \ref{S2.sec}. For now, we consider just one
example, the computation of the third derivative of $z_{j}(\tau)$ in the (S0) or
S1 case. We start by applying Theorem \ref{secDeriv.thm} to the computation of
$(z_{j}^{N})^{\prime\prime}$ and then use Theorem \ref{S0deriv.thm} or Theorem
\ref{S1deriv.thm} to compute one more derivative:
\begin{align}
z_{j}^{\prime\prime\prime}(\tau_{0})  &  =6\lim_{N\rightarrow\infty}\sum
_{k=1}^{N}\frac{\mathbf{1}_{\{k\neq j\}}}{(z_{j}(\tau_{0})-z_{k}(\tau
_{0}))^{4}}\left(  (z_{j}^{N})^{\prime}(\tau_{0})-(z_{k}^{N})^{\prime}%
(\tau_{0})\right) \label{zTriple1}\\
&  =6\lim_{N\rightarrow\infty}\sum_{k=1}^{N}\frac{\mathbf{1}_{\{k\neq j\}}%
}{(z_{j}(\tau_{0})-z_{k}(\tau_{0}))^{4}}\nonumber\\
&  \times\sum_{l=1}^{N}\left(  \frac{\mathbf{1}_{\{l\neq j\}}}{z_{j}(\tau
_{0})-z_{l}(\tau_{0})}-\frac{\mathbf{1}_{\{l\neq k\}}}{z_{k}(\tau_{0}%
)-z_{l}(\tau_{0})}\right)  , \label{zTriple2}%
\end{align}
where in the (S1) case, we note that the $a_{1}(\tau)$ and $1/(w-c)$ terms in
the formulas for $(z_{j}^{N})^{\prime}$ and $(z_{k}^{N})^{\prime}$ cancel.

On the right-hand sides of (\ref{zTriple1}) and (\ref{zTriple2}), we must let
both occurrences of $N$ tend to infinity simultaneously. In (\ref{zTriple1}),
for example, it is \textit{not} correct to let the $N$ in $(z_{j}^{N}%
)^{\prime}$ and $(z_{k}^{N})^{\prime}$ tend to infinity first and then let the
$N$ in the sum tend to infinity, which would give%
\begin{equation}
z_{j}^{\prime\prime\prime}(\tau_{0})\overset{??}{=}6\sum_{k=1}^{\infty}%
\frac{\mathbf{1}_{\{k\neq j\}}}{(z_{j}(\tau_{0})-z_{k}(\tau_{0}))^{4}}\left(
z_{j}^{\prime}(\tau_{0})-z_{k}^{\prime}(\tau_{0})\right)  . \label{sumDiverge}%
\end{equation}
The sum on the right-hand side of (\ref{sumDiverge}) may not converge, because
the quantities $z_{k}^{\prime}(\tau_{0})$ can grow very rapidly as a function
of $k.$

On the other hand, the formula (\ref{MNprime}) in Theorem
\ref{momentDeriv.thm} will show that the finite sums on the right-hand side of
(\ref{zTriple2}) can be simplified to give the following expression for
$z_{j}^{\prime\prime\prime}(\tau_{0})$:%
\begin{align*}
z_{j}^{\prime\prime\prime}(\tau_{0})  &  =\lim_{N\rightarrow\infty}\left[
18\sum_{k=1}^{N}\frac{\mathbf{1}_{\{k\neq j\}}}{(z_{j}(\tau_{0})-z_{k}%
(\tau_{0}))^{5}}\right. \\
&  -6\sum_{k=1}^{N}\frac{\mathbf{1}_{\{k\neq j\}}}{(z_{j}(\tau_{0})-z_{k}%
(\tau_{0}))^{2}}\left.  \sum_{k=1}^{N}\frac{\mathbf{1}_{\{k\neq j\}}}%
{(z_{j}(\tau_{0})-z_{k}(\tau_{0}))^{3}}\right]  .
\end{align*}
At this point, we can evaluate the limit to obtain an expression for
$z_{j}^{\prime\prime\prime}(\tau_{0})$ involving only sums that are convergent
in the (S0) or (S1) case:%
\begin{align*}
z_{j}^{\prime\prime\prime}(\tau_{0})  &  =18\sum_{k=1}^{\infty}\frac
{\mathbf{1}_{\{k\neq j\}}}{(z_{j}(\tau_{0})-z_{k}(\tau_{0}))^{5}}\\
&  -6\sum_{k=1}^{\infty}\frac{\mathbf{1}_{\{k\neq j\}}}{(z_{j}(\tau_{0}%
)-z_{k}(\tau_{0}))^{2}}\sum_{k=1}^{\infty}\frac{\mathbf{1}_{\{k\neq j\}}%
}{(z_{j}(\tau_{0})-z_{k}(\tau_{0}))^{3}}.
\end{align*}

\begin{proof}
[Proof of Theorem \ref{trunc.thm}]It is, in principle, possible to compute all
higher derivatives of $z_{j}(\tau),$ $a_{1}(\tau),$ and $a_{2}(\tau)$ at
$\tau_{0},$ by the same method as in the proofs of the theorems in Section
\ref{FirstSecond.sec}. The result will always be an expression that is
continuous with respect to locally uniform convergence.

In the case of $z_{j}(\cdot),$ for example, we may apply $m$ derivatives in
$\tau$ to the identity $F(z_{j}(\tau),\tau)=0$ and then evaluate at $\tau
=\tau_{0}.$ We will get exactly one term involving $z_{j}^{(m)}(\tau_{0}),$
namely%
\[
F_{z}(z_{j}(\tau_{0}),\tau_{0})z_{j}^{(m)}(\tau_{0}).
\]
All the other terms will be combinations of derivatives of $F$ and lower
derivatives of $z_{j}.$ Working inductively, we will get an expression of the
form%
\[
z_{j}^{(m)}(\tau_{0})=\left.  P_{m}\left(  \frac{F_{zz}}{F_{z}},\ldots
,\frac{F^{(2m)}}{F_{z}}\right)  \right\vert _{\tau=\tau_{0},~z=z_{j}(\tau
_{0}),},
\]
where $P_{m}$ is a polynomial. (Compare (\ref{zDouble2}) in the case $m=2.$)

When we replace $F$ by $F^{N},$ we appeal to the locally uniform convergence
of the Hadamard product, which implies convergence of all derivatives. Since
the zeros at time $\tau_{0}$ are simple, $F_{z}(z_{j}(\tau_{0}),\tau_{0})$ and
$F_{z}^{N}(z_{j}(\tau_{0}),\tau_{0})$ are nonzero and we conclude that
$(z_{j}^{N})^{(m)}(\tau_{0})$ converges to $z_{j}^{(m)}(\tau_{0}).$
\end{proof}

\subsection{Case (S1)\label{S0S1.sec}}

In this section, we apply Theorem \ref{trunc.thm} in the (S1) case, including the case (S0). We
give an algorithm by which we can compute all the higher derivatives
$z_{j}^{(m)}(\tau_{0})$ explicitly as functions of the quantities
$\{z_{k}(\tau_{0})\}_{k=1}^{\infty}.$ %The formulas for the derivatives with$m\geq2$ are the same in the (S0) and (S1) case.
In particular, the formulas for
the higher derivatives of $z_{j}$ in the (S1) case do not involve $a_{1}%
(\tau_{0}).$
%explicitly as possible. ???
The formulas for these derivatives will be the same in
the (S0) and (S1) cases, except when $m=1,$ where there are two extra terms in
Theorem \ref{S1deriv.thm} as compared to Theorem \ref{S0deriv.thm}.

\begin{definition}
\label{mj.def}For each $p\geq2,$ and each $j,$ we define%
\[
M^{N}(j,p,\tau)=\sum_{k=1}^{N}\frac{\mathbf{1}_{\{j\neq k\}}}{(z_{j}%
(\tau)-z_{k}(\tau))^{p}}.
\]
We then define an $N=\infty$ version of $M^{N},$ but only evaluated at
$\tau=\tau_{0}$:%
\[
M(j,p)=\sum_{k=1}^{\infty}\frac{\mathbf{1}_{\{j\neq k\}}}{(z_{j}(\tau
_{0})-z_{k}(\tau_{0}))^{p}},
\]
where the series is convergent for $p\geq2$ in the (S0) and (S1) cases.
\end{definition}

Note that in the notation of Definition \ref{mj.def}, Theorem
\ref{secDeriv.thm} (applied to $z_{j}^{N}(\tau)$) tells us that%
\begin{equation}
(z_{j}^{N})^{\prime\prime}(\tau)=-2M^{N}(j,3,\tau). \label{zDoubleM3}%
\end{equation}

\begin{theorem}
\label{momentDeriv.thm}In the (S0) and (S1) cases (Theorems \ref{S0deriv.thm} and
\ref{S1deriv.thm}), we have
\begin{align}
&  M^{N}(j,p,\tau)^{\prime}\nonumber\\
&  =-\frac{p}{2}\left(  (p+3)M^{N}(j,p+2,\tau)-\sum_{n=2}^{p}M^{N}%
(j,p+2-n,\tau)M^{N}(j,n,\tau)\right)  , \label{MNprime}%
\end{align}
where the prime indicates differentiation with respect to $\tau.$ We may then
apply this result repeatedly to compute $(z_{j}^{N})^{(m)}(\tau_{0}),$
starting from (\ref{zDoubleM3}). After applying Theorem \ref{trunc.thm}, we
will then obtain an expression of the form%
\[
z_{j}^{(m)}(\tau_{0})=Q_{m}(M(j,2),\ldots,M(j,2m-1)),
\]
where $Q_{m}$ is a polynomial. In particular, we have%
\begin{equation}
z_{j}^{\prime\prime\prime}(\tau_{0})=18M(j,5)-6M(j,2)M(j,3). \label{zTripleS1}%
\end{equation}

\end{theorem}

\begin{proof}
It is convenient to use the notation%
\[
\gamma_{jk}=1-\delta_{jk}=\left\{
\begin{array}
[c]{rr}%
1 & j\neq k\\
0 & j=k
\end{array}
\right.  .
\]
Since $N$ is fixed and finite throughout the proof, we omit the superscript
\textquotedblleft$N$\textquotedblright\ on the variables and we assume that
all sums over indices $k$ and $l$ will range from 1 to $N.$ We also do not
explicitly indicate the dependence of each variable on $\tau.$

We these conventions, we begin by computing $z_{j}^{\prime}-z_{k}^{\prime}$ in
the (S0) and (S1) cases using Theorem \ref{S0deriv.thm} and \ref{S1deriv.thm}. The
result is the same in both cases, because the extra terms in Theorem
\ref{S1deriv.thm} are the same for $z_{j}^{\prime}$ as for $z_{k}^{\prime}.$
We then get%
\begin{align}
z_{j}^{\prime}-z_{k}^{\prime}  &  =\sum_{l}\frac{\gamma_{lj}}{z_{j}-z_{l}%
}-\sum_{l}\frac{\gamma_{lk}}{z_{k}-z_{l}}\nonumber\\
&  =\frac{1}{z_{j}-z_{k}}+\sum_{l}\frac{\gamma_{lj}\gamma_{lk}}{z_{j}-z_{l}%
}-\frac{1}{z_{k}-z_{j}}-\sum_{l}\frac{\gamma_{lk}\gamma_{lj}}{z_{k}-z_{l}%
}\nonumber\\
&  =\frac{2}{z_{j}-z_{k}}+\sum_{l}\gamma_{lj}\gamma_{lk}\left(  \frac{1}%
{z_{j}-z_{l}}-\frac{1}{z_{k}-z_{l}}\right)  . \label{DerivativeDiff.eq}%
\end{align}

We then compute, using (\ref{DerivativeDiff.eq}), that
\begin{align}
\frac{d}{d\tau}M^{N}(j,p)  &  =-p\sum_{k}\frac{\gamma_{jk}}{(z_{j}%
-z_{k})^{p+1}}(z_{j}^{\prime}-z_{k}^{\prime})\nonumber\\
&  =-2p\sum_{k}\frac{\gamma_{jk}}{(z_{j,N}-z_{k,N})^{p+2}}\nonumber\\
&  +p\sum_{k,l}\frac{\gamma_{jk}\gamma_{jl}\gamma_{kl}}{(z_{j,N}-z_{k,N})^{p}%
}\frac{1}{(z_{j,N}-z_{l,N})(z_{k,N}-z_{l,N})} \label{SumDerivativeMoments.eq}%
\end{align}
In the second sum on the right-hand side of (\ref{SumDerivativeMoments.eq}),
we symmetrize the summand with respect to $k$ and $l$:
\begin{align*}
&  \sum_{k,l}\frac{\gamma_{jk}\gamma_{jl}\gamma_{kl}}{(z_{j,N}-z_{k,N})^{p}%
}\frac{1}{(z_{j,N}-z_{l,N})(z_{k,N}-z_{l,N})}\\
&  =\frac{1}{2}\sum_{k,l}\gamma_{jk}\gamma_{jl}\gamma_{kl}\frac{(z_{j}%
-z_{l})^{p-1}-(z_{j}-z_{k})^{p-1}}{(z_{j}-z_{k})^{p}(z_{j}-z_{l})^{p}%
(z_{k}-z_{l})},
\end{align*}
which further simplifies to
\begin{align}
&  \sum_{k,l}\frac{\gamma_{jk}\gamma_{jl}\gamma_{kl}}{(z_{j,N}-z_{k,N})^{p}%
}\frac{1}{(z_{j,N}-z_{l,N})(z_{k,N}-z_{l,N})}\nonumber\\
&  =\frac{1}{2}\sum_{k,l}\gamma_{jk}\gamma_{jl}\gamma_{kl}\frac{(z_{j}%
-z_{l})^{p-2}+(z_{j}-z_{l})^{p-3}(z_{j}-z_{l})+\ldots+(z_{j}-z_{k})^{p-2}%
}{(z_{j}-z_{k})^{p}(z_{j}-z_{l})^{p}}. \label{doubleSum}%
\end{align}

Now, in each sum on the right-hand side of (\ref{doubleSum}), the factor of
$\gamma_{kl}$ means that the term with $l=k$ is omitted. Since there is now no
factor of $z_{k}-z_{l}$ in the denominator, we can add and subtract a term
with $l=k.$ Thus, we get a double sum without the factor of $\gamma_{kl}$
minus a single sum with $l=k.$ The single sums are actually all the same and
we obtain
\begin{align*}
&  \sum_{k,l}\frac{\gamma_{jk}\gamma_{jl}\gamma_{kl}}{(z_{j,N}-z_{k,N})^{p}%
}\frac{1}{(z_{j,N}-z_{l,N})(z_{k,N}-z_{l,N})}\\
&  =\frac{1}{2}\sum_{n=2}^{p}\sum_{k}\frac{\gamma_{jk}}{(z_{j}-z_{k})^{p-n+2}%
}\sum_{k}\frac{\gamma_{jk}}{(z_{j}-z_{k})^{n}}-\frac{p-1}{2}\sum_{k}%
\frac{\gamma_{jk}}{(z_{j}-z_{k})^{p+2}}.
\end{align*}
Putting this computation into (\ref{SumDerivativeMoments.eq}) and recalling
the definition of $M^{N}(j,p,\tau)$ gives (\ref{MNprime}).
\end{proof}

\subsection{Case (S2)\label{S2.sec}}

In the (S2) case, we again apply Theorem \ref{trunc.thm}; that is, we compute
derivatives of $z_{j}^{N}(\tau)$ at $\tau=\tau_{0}$ and then let $N$ tend to
infinity. To compute the derivatives of $z_{j}^{N}(\tau),$ we consider the
function $F^{N}(z,\tau)$ in (\ref{FNdef}) as a truncated version of the
Hadamard product in (\ref{Hadamard3}). We can then write the truncated product
in the form%
\[
F^{N}(z,\tau)=\exp\left\{  a_{0}^{N}+b(\tau)z+\frac{1}{2}a(\tau)z^{2}\right\}
p^{N}(z,\tau),
\]
where%
\begin{align}
b^{N}(\tau)  &  =a_{1}^{N}(\tau)+\sum_{k=1}^{N}\frac{1}{z_{k}^{N}(\tau
)}\label{bdef}\\
a^{N}(\tau)  &  =a_{2}^{N}(\tau)+\sum_{k=1}^{N}\frac{1}{z_{k}^{N}(\tau)^{2}%
}\label{adef}\\
p^{N}(z,\tau)  &  =\prod_{k}\left(  1-\frac{z}{z_{k}^{N}(\tau)}\right)  .
\label{pdef}%
\end{align}

We then let $\{z_{j}^{N}(\tau)\}_{j=1}^{N}$ denote the zeros of $F^{N}%
(z,\tau)$ and $\{Z_{j}^{N}(\tau)\}_{j=1}^{N}$ denote the zeros of
$p^{N}(z,\tau).$ Using Proposition \ref{prop:heat_flow_and_mult_by_exp} and
the definition (\ref{FNdef}) of $F^{N},$ we find that%
\begin{equation}
z_{j}^{N}(\tau)=(\tau-\tau_{0})b^{N}(\tau_{0})+(1+(\tau-\tau_{0})a^{N}%
(\tau_{0}))Z_{j}^{N}\left(  \frac{\tau-\tau_{0}}{1+(\tau-\tau_{0})a^{N}%
(\tau_{0})}\right)  , \label{zjAndZj}%
\end{equation}
where $b^{N}$ and $a^{N}$ are as in (\ref{bdef}) and (\ref{adef}),
respectively. We can then relate derivatives of $z_{j}(\tau)$ at $\tau
=\tau_{0}$ to the corresponding derivatives of $Z_{j}(\tau).$ We compute just
one example explicitly, to illustrate the method.

\begin{proposition}
In the (S2) case, we have%
\begin{align*}
z_{j}^{\prime\prime\prime}(\tau_{0})  &  =-3a_{2}(\tau_{0})-12\sum
_{k=1}^{\infty}\frac{\mathbf{1}_{\{k\neq j\}}}{(z_{j}(\tau_{0})-z_{k}(\tau
_{0}))^{5}}\\
&  +6\sum_{k=1}^{\infty}\left(  \frac{\mathbf{1}_{\{k\neq j\}}}{(z_{j}%
(\tau_{0})-z_{k}(\tau_{0}))^{2}}-\frac{1}{z_{k}(\tau_{0})^{2}}\right)  \left(
\sum_{k=1}^{\infty}\frac{1}{(z_{j}(\tau_{0})-z_{k}(\tau_{0}))^{3}}\right)  .
\end{align*}

\end{proposition}

It is interesting to compare this result to the computation of $z_{j}%
^{\prime\prime\prime}(\tau_{0})$ in the (S0) and (S1) cases, in (\ref{zTripleS1}).
The formula in the (S2) case differs by the presence of the leading term
involving $a_{2}(\tau_{0})$ and the replacement of $M(j,2)$ by a regularized
version of this sum.

\begin{proof}
We first compute using (\ref{zjAndZj}) that%
\[
(z_{j}^{N})^{\prime\prime\prime}=-3a^{N}(\tau_{0})(Z_{j}^{N})^{\prime\prime
}+(Z_{j}^{N})^{\prime\prime\prime}.
\]
We then compute the derivatives of $Z_{j}^{N}$ using results from the (S0) case,
to obtain%
\begin{align*}
(z_{j}^{N})^{^{\prime\prime\prime}}  &  =-3\left(  a_{2}^{N}(\tau_{0}%
)+\sum_{k=1}^{N}\frac{1}{z_{k}(\tau_{0})^{2}}\right)  2M^{N}(j,3,\tau_{0})\\
&  -12M^{N}(j,5,\tau_{0})+6M^{N}(j,2,\tau_{0})M^{N}(j,3,\tau_{0}),
\end{align*}
which simplifies to
\begin{align*}
(z_{j}^{N})^{^{\prime\prime\prime}}  &  =-3a_{2}^{N}(\tau_{0})-12M^{N}%
(j,5,\tau_{0})\\
&  +6M^{N}(j,3,\tau_{0})\sum_{k=1}^{N}\left(  \frac{\mathbf{1}_{\{k\neq j\}}%
}{(z_{j}(\tau_{0})-z_{k}(\tau_{0}))^{2}}-\frac{1}{z_{k}(\tau_{0})^{2}}\right)
.
\end{align*}
Letting $N$ tend to infinity gives the claimed result.
\end{proof}

\appendix{}

\section{The metaplectic representation\label{meta.appendix}}

In this appendix, we review some results about the metaplectic representation.
We mostly follow Chapter 4 of the book \cite{folland_book}, in the case $n=1.$

\addtocontents{toc}{\protect\setcounter{tocdepth}{1}}
\subsection{The group $SL(2;\mathbb{R})$ in classical mechanics}

In this section, we give a very brief review of certain parts of classical
mechanics, for a particle moving in $\mathbb{R}.$ See, for example,
\cite[Chapter 2]{HallQM} for more information. The construction described here
provide motivation for the construction of the metaplectic representation in
subsequent sections.

When studying the classical mechanics of a particle moving in $\mathbb{R}$,
one considers the \textquotedblleft phase space\textquotedblright%
\ $\mathbb{R}^{2},$ thought of as the set of pairs $(x,p)$, with $x$ being the
position of the particle and with $p$ being the associated momentum. Given a
smooth \textquotedblleft Hamiltonian\textquotedblright\ function
$H:\mathbb{R}^{2}\rightarrow\mathbb{R},$ we define a time-evolution in
$\mathbb{R}^{2}$ by Hamilton's equations:%
\begin{equation}
\frac{dx}{ds}=\frac{\partial H}{\partial p}(x(s),p(s));\quad\frac{dp}%
{ds}=-\frac{\partial H}{\partial x}(x(s),p(s)). \label{HamEq}%
\end{equation}
If, for example,
\[
H(x,p)=\frac{p^{2}}{2m}+V(x),
\]
where $m$ is the mass of the particle and $V$ is the potential energy, we get%
\[
\frac{dx}{ds}=\frac{p}{m};\quad\frac{dp}{ds}=-\frac{\partial V}{\partial x}.
\]
Then, by differentiating the formula for $dx/ds$ with respect to $s,$ we get%
\[
m\frac{d^{2}x}{ds^{2}}=-V^{\prime}(x),
\]
which is equivalent to Newton's second law $F=ma,$ with the force $F$ give by
$F(x)=-V^{\prime}(x).$

In the general setting, if we define the \textquotedblleft Poisson
bracket\textquotedblright\ of two smooth functions on $\mathbb{R}^{2}$ by
\[
\{f,g\}=\frac{\partial f}{\partial x}\frac{\partial g}{\partial p}%
-\frac{\partial f}{\partial p}\frac{\partial g}{\partial x},
\]
then we have the basic identity that for any smooth function $f$ and any
solution $(x(s),p(s))$ of Hamilton's equations, we have%
\[
\frac{d}{ds}f(x(s),p(s))=\{f,H\}(x(s),p(s)).
\]
The proof is just the chain rule:%
\[
\frac{df}{ds}=\frac{\partial f}{\partial x}\frac{dx}{ds}+\frac{\partial
f}{\partial p}\frac{dp}{ds}=\frac{\partial f}{\partial x}\frac{\partial
H}{\partial p}-\frac{\partial f}{\partial p}\frac{\partial H}{\partial
x}=\{f,H\}.
\]

The most basic example of the Poisson bracket is this:%
\[
\{x,p\}=1.
\]
We observe that the space of homogeneous polynomials of degree 2---spanned by
$\frac{1}{2}p^{2},$ $\frac{1}{2}x^{2},$ and $xp$---forms a Lie algebra under
the Poisson bracket:%
\begin{align*}
\left\{  xp,\frac{1}{2}p^{2}\right\}   &  =2\left(  \frac{1}{2}p^{2}\right) \\
\left\{  xp,\frac{1}{2}x^{2}\right\}   &  =-2\left(  \frac{1}{2}x^{2}\right)
\\
\left\{  \frac{1}{2}x^{2},\frac{1}{2}p^{2}\right\}   &  =xp
\end{align*}
and this Lie algebra is isomorphic to $sl(2;\mathbb{R}).$ The commutations
relations for the preceding basis elements are the same as the relations for
the standard basis of $sl(2;\mathbb{R)}$:%
\[
\left(
\begin{array}
[c]{rr}%
0 & 1\\
0 & 0
\end{array}
\right)  ;\quad\left(
\begin{array}
[c]{rr}%
0 & 0\\
1 & 0
\end{array}
\right)  ;\quad\left(
\begin{array}
[c]{rr}%
1 & 0\\
0 & -1
\end{array}
\right)  .
\]
The connection to the group $SL(2;\mathbb{R})$ is not coincidental, as we now demonstrate.

If $f$ is a smooth function on $\mathbb{R}^{2},$ we define the
\textbf{Hamiltonian flow} $\Phi_{f}$ associated to $f$ to be the flow on
$\mathbb{R}^{2}$ obtained by solving Hamilton's equations (\ref{HamEq}) with
Hamiltonian $H=-f.$ (The minus sign ensures that if $X_{f}$ is the vector
field generating $\Phi_{f}$---the infinitesimal flow---then $[X_{f}%
,X_{g}]=X_{\{f,g\}}.$) If $f$ is a homogeneous polynomial of degree 2 in $x$
and $p,$ then Hamilton's equations are linear and $\Phi_{f}(s)$ will be a
one-parameter group of linear transformations with determinant 1.

We should note at this point that $SL(2;\mathbb{R})$ is equal to the group of
$2\times2$ matrices preserving the skew-symmetric bilinear form $\omega$ given
by%
\[
\omega((x_{1},y_{1}),(x_{2},y_{2}))=x_{1}y_{2}-y_{1}x_{2},
\]
which is the \textbf{symplectic group} of $\mathbb{R}^{2},$ denoted either
$Sp(1;\mathbb{R})$ or $Sp(2;\mathbb{R})$ depending on the author. In
$\mathbb{R}^{2n},$ for $n\geq2,$ the homogeneous polynomials of degree $2$
form a Lie algebra under the natural Poisson bracket---and this Lie algebra is
isomorphic to the Lie algebra of the symplectic group of $\mathbb{R}^{2n}$ and
\textit{not} to the Lie algebra of $SL(2n;\mathbb{R}).$

We consider two basic examples.

\begin{example}
\label{classicalRot.ex}If we consider $f(x,p)=\frac{1}{2}(x^{2}+p^{2}),$ then
Hamilton's equations (\ref{HamEq}) with Hamiltonian $-f$ read%
\[
\frac{dx}{ds}=-p;\quad\frac{dp}{ds}=x.
\]
The solutions are easily verified to be:
\begin{align*}
x(s)  &  =\cos(s)x_{0}-\sin(s)p_{0}\\
p(s)  &  =\sin(s)x_{0}+\cos(s)p_{0}.
\end{align*}
Thus, the Hamiltonian flow $\Phi_{f}$ is given by counter-clockwise rotations
of the initial position and momentum:%
\begin{equation}
\left(
\begin{array}
[c]{r}%
x(s)\\
p(s)
\end{array}
\right)  =\left(
\begin{array}
[c]{rr}%
\cos s & -\sin s\\
\sin s & \cos s
\end{array}
\right)  \left(
\begin{array}
[c]{r}%
x_{0}\\
p_{0}%
\end{array}
\right)  . \label{ClassicalRotate}%
\end{equation}

\end{example}

\begin{example}
\label{classicalDil.ex}If we consider $f(x,p)=xp,$ then Hamilton's equations
(\ref{HamEq}) with Hamiltonian $-f$ read
\[
\frac{dx}{ds}=-x;\quad\frac{dp}{ds}=p.
\]
The solutions are easily verified to be:%
\begin{align*}
x(s)  &  =x_{0}e^{-s}\\
p(s)  &  =p_{0}e^{s}.
\end{align*}
Thus, the Hamiltonian flow $\Phi_{f}$ is given by the action of the positive
diagonal subgroup of $SL(2;\mathbb{R)}$:%
\begin{equation}
\left(
\begin{array}
[c]{r}%
x(s)\\
p(s)
\end{array}
\right)  =\left(
\begin{array}
[c]{rr}%
e^{-s} & 0\\
0 & e^{s}%
\end{array}
\right)  \left(
\begin{array}
[c]{r}%
x_{0}\\
p_{0}%
\end{array}
\right)  . \label{diagClassical}%
\end{equation}

\end{example}

\subsection{The group $SL(2;\mathbb{R})$ in quantum
mechanics\label{metaL2R.sec}}

In quantum mechanics for a particle moving in $\mathbb{R}$, one builds a
Hilbert space $\mathbf{H}$ (the \textquotedblleft quantum Hilbert
space\textquotedblright) and then attempts to associate to each smooth
function $f$ on $\mathbb{R}^{2}$ a self-adjoint operator $Q(f)$ on
$\mathbf{H}.$ (See Chapters 3 and 13 of \cite{HallQM}.) The map $Q$ from
functions to operators is supposed to satisfy%
\begin{equation}
Q(1)=I\text{ (the identity operator)} \label{QofOne}%
\end{equation}
and, as much as possible, the relation%
\begin{equation}
Q(\{f,g\})=-i[Q(f),Q(g)], \label{comm}%
\end{equation}
where $[\cdot,\cdot]$ is the commutator of operators, defined by
\[
\lbrack A,B]=AB-BA.
\]
(In (\ref{comm}), we work in units where Planck's constant equals 1.) As it
turns out, there is no \textquotedblleft reasonable\textquotedblright\ map $Q$
satisfying (\ref{QofOne}) and (\ref{comm}); see Section 13.4 of \cite{HallQM}.
Nevertheless, we hope to achieve (\ref{QofOne}) and (\ref{comm}) on
\textit{some} space of functions, which in our case will be the space of
polynomials in $p$ and $q$ of degree at most 2.

Typically, the Hilbert space $\mathbf{H}$ is taken to be $L^{2}(\mathbb{R})$
and we begin by defining
\begin{align}
Q(1)  &  =I\label{1op}\\
Q(x)  &  =X:=\text{multiplication by }x\label{Xop}\\
Q(p)  &  =P:=-i\frac{d}{dx}. \label{Pop}%
\end{align}
We check that
\[
-i[X,P]=I.
\]
We then define%
\begin{equation}
Q(x^{2})=X^{2};\quad Q(p^{2})=P^{2};\quad Q(xp)=\frac{1}{2}(XP+PX).
\label{3ops}%
\end{equation}
Note that we cannot take $Q(xp)=XP$ because $XP$ is not self-adjoint. But
\begin{equation}
\frac{1}{2}(XP+PX)=-ix\frac{d}{dx}-\frac{i}{2} \label{xppx}%
\end{equation}
is self-adjoint.

It is then an elementary computation to verify that the commutation relations
in (\ref{comm}) do hold (on, say, the Schwarz space inside $L^{2}(\mathbb{R}%
)$), for all polynomials in $x$ and $p$ of degree at most 2, that is, on the
span of $\{1,x,p,x^{2},xp,p^{2}\}.$ Things do not work out so nicely, however,
if we try to quantize higher-degree polynomials. See, again, Section 13.4 of
\cite{HallQM}.

Meanwhile, if $A$ is a (possibly unbounded) self-adjoint operator, Stone's
theorem says that we can form a one-parameter unitary group $U_{A}$ given by%
\[
U_{A}(s)=e^{-isA},
\]
so that $A$ can be recovered from $U_{A}(s)$ as%
\[
A\psi=i\left.  \frac{d}{ds}U_{A}(s)\psi\right\vert _{s=0},
\]
for all $\psi$ in the domain of $A.$ See \cite[Section 10.2]{HallQM}.

\begin{theorem}
Let $G$ be the connected double cover of $SL(2;\mathbb{R}),$ and let us
identify the Lie algebra of $G$ with the Lie algebra of $SL(2;\mathbb{R})$,
which we in turn identify with the space of homogeneous polynomials of degree
$2$ in $x$ and $p.$ Then there is a unique unitary representation $U(\cdot)$
of $G$ acting on $L^{2}(\mathbb{R})$ such that
\[
U(e^{sf})=e^{-isQ(f)}%
\]
for all $f$ in the Lie algebra of $G.$ We call $U$ the \textbf{metaplectic
representation} of $G.$
\end{theorem}

See, for example, Chapter 4 of Folland's book \cite{folland_book}, in the case
$n=1.$

\subsection{The metaplectic representation in the Segal--Bargmann space}

In Section \ref{meta.sec}, we described the metaplectic representation on the
Segal--Bargmann space using integral operators. Here we use a Lie algebra
approach to give a more direct construction, which may provide motivation for
some of the key formulas in

In Section \ref{metaL2R.sec}, we used the conventional choice of the quantum
Hilbert space $\mathbf{H}$, namely $L^{2}(\mathbb{R}),$ consisting of
functions of a position variable $x.$ But we can also use the Segal--Bargmann
space $\mathcal{B}$ introduced in (\ref{eq:SBdef}). There is a natural unitary
map $B$ from $L^{2}(\mathbb{R})$ to $\mathcal{B}$ known as the
\textbf{Segal--Bargmann transform}. (See the paper \cite{Bargmann} of Bargmann
or Section 6 of \cite{hall_holomorphic_methods}.) We may then conjugate all
the constructions in the previous section by the Segal--Bargmann transform to
get a metaplectic representation in $\mathcal{B}.$ This means that we
introduce a new quantization map $Q_{\mathrm{SB}}$ related to the original map
$Q$ by%
\[
Q_{\mathrm{SB}}(f)=BQ(f)B^{-1}.
\]

To see how this works, one typically introduces the creation and annihilation
operators $a$ and $a^{\ast}$ in $L^{2}(\mathbb{R})$ given by%
\[
a=\frac{X+iP}{\sqrt{2}};\quad a^{\ast}=\frac{X-iP}{\sqrt{2}}.
\]
Conjugating $a$ and $a^{\ast}$ by the Segal--Bargmann transform, we get
operators $A$ and $A^{\ast}$ given by%
\[
A=\frac{d}{dz};\quad A^{\ast}=z.
\]
That is, $A^{\ast}$ consists of multiplication by $z$. (See, again,
\cite{Bargmann} or \cite[Section 6]{hall_holomorphic_methods}.)

We then consider our two main examples, corresponding to the functions
$\frac{1}{2}(x^{2}+p^{2})$ and $xp.$ In $L^{2}(\mathbb{R}),$ we compute that%
\begin{align*}
\frac{1}{2}(X^{2}+P^{2})  &  =a^{\ast}a+\frac{1}{2}\\
\frac{1}{2}(XP+PX)  &  =\frac{i}{2}((a^{\ast})^{2}-a^{2}).
\end{align*}
Thus, in the Segal--Bargmann space, we may consider the corresponding
self-adjoint operators obtained by replacing $a$ by $A$ and $a^{\ast}$ by
$A^{\ast}$:%
\begin{align}
Q_{\mathrm{SB}}\left(  \frac{1}{2}(x^{2}+p^{2})\right)   &  =A^{\ast}%
A+\frac{1}{2}=z\frac{d}{dz}+\frac{1}{2};\label{SBx2p2}\\
Q_{\mathrm{SB}}(xp)  &  =\frac{i}{2}((A^{\ast})^{2}-A^{2})=\frac{i}{2}\left(
z^{2}-\frac{d^{2}}{dz^{2}}\right)  . \label{sbxp}%
\end{align}

We now consider the quantum counterparts of Examples \ref{classicalRot.ex} and
\ref{classicalDil.ex}.

\begin{example}
\label{MetaRot.ex}Let us the function $f(x,p)=\frac{1}{2}(x^{2}+p^{2}),$ for
which the classical Hamiltonian flow $\Phi_{f}(s)$ consists of
counter-clockwise rotations by angle $s.$ We then let $Q_{\mathrm{SB}}(f)$ be
the quantum operator in the Segal--Bargmann space associated to $f,$ as given
by (\ref{SBx2p2}). Then we have%
\begin{equation}
(e^{-isQ_{\mathrm{SB}}(f)}F)(z)=e^{-is/2}F(e^{-is}z) \label{expOsc}%
\end{equation}
for every holomorphic function $F$ in the Segal--Bargmann space.
\end{example}

We note that when $s=2\pi,$ the operator in (\ref{expOsc}) is equal to $-I,$
even though the matrices in (\ref{ClassicalRotate}) come back to the $I$ at
$s=2\pi.$ This discrepancy indicates that the metaplectic representation
cannot be defined as a representation of $SL(2;\mathbb{R})$ itself, but only
as a representation of its connected double cover.

\begin{proof}
It is evident that the right-hand side of (\ref{expOsc}) defines a
one-parameter unitary group on the Segal--Bargmann space, and we easily verify
that
\[
\frac{\partial}{\partial s}e^{-is/2}F(e^{-is}z)=-i\left(  z\frac{\partial
}{\partial z}+\frac{1}{2}\right)  e^{-is/2}F(e^{-is}z).
\]
Thus, (\ref{expOsc}) holds.
\end{proof}

\begin{example}
\label{MetaDiag.ex}Let us consider the function $f(x,p)=xp,$ for which the
classical Hamiltonian flow $\Phi_{f}(s)$ is given by diagonal matrices with
diagonal entries $e^{-s}$ and $e^{s}.$ We then let $Q_{\mathrm{SB}}(f)$ be the
quantum operator associated to $f,$ as given by (\ref{sbxp}). Then we have%
\begin{equation}
(e^{-isQ_{\mathrm{SB}}(f)}F)(z)=\left(  1-t^{2}\right)  ^{1/4}~e^{\frac{t}%
{2}z^{2}}\left(  e^{-\frac{t}{2}\frac{d^{2}}{dz^{2}}}F\right)  \left(
z\sqrt{1-t^{2}}\right)  ,\quad s\in\mathbb{R},\label{uzt}%
\end{equation}
where%
\[
t=\tanh s,
\]
for every holomorphic function $F$ in the Segal--Bargmann space.
\end{example}

The operator on the right-hand side of (\ref{uzt}) is just the operator
$V_{\tau}$ in (\ref{eq:Vtau}), in the case that $\tau$ is a real number of the
form $\tau=t=-\tanh s.$

\begin{proof}
We compute formally, that is, without worrying about the domains. The result
is closely related to the Mehler formula in Example \ref{mehler.ex}, which can
be understood as a formula for the propagator of the quantum harmonic
oscillator. The quantum operator should satisfy%
\begin{align*}
\frac{d}{ds}(e^{-isQ_{\mathrm{SB}}(f)}F)(z)  &  =-i[Q_{\mathrm{SB}}(f)F](z)\\
&  =\frac{1}{2}\left(  z^{2}-\frac{d^{2}}{dz^{2}}\right)  [Q_{\mathrm{SB}%
}(f)F](z).
\end{align*}
To verify this relation, we let $u(z,t)=(e^{-\frac{t}{2}D^{2}/2}F)(z),$ so
that the right-hand side of (\ref{uzt}) is given by
\[
\left(  1-t^{2}\right)  ^{1/4}~e^{\frac{t}{2}z^{2}}u\left(  z\sqrt{1-t^{2}%
},t\right)  .
\]
We note that%
\[
\frac{\partial}{\partial s}=(1-t^{2})\frac{\partial}{\partial t}%
\]
and then compute that
\begin{align*}
&  \left(  \frac{\partial}{\partial s}-\frac{1}{2}\left(  z^{2}-\frac
{\partial^{2}}{\partial z^{2}}\right)  \right)  \left(  1-t^{2}\right)
^{1/4}~e^{\frac{t}{2}z^{2}}u\left(  z\sqrt{1-t^{2}},t\right) \\
&  =-\frac{1}{2}(1-t^{2})^{5/4}e^{\frac{t}{2}z^{2}}\left(  \frac{\partial
u}{\partial t}\left(  z\sqrt{1-t^{2}},t\right)  +\frac{1}{2}\frac{\partial
^{2}u}{\partial x^{2}}\left(  z\sqrt{1-t^{2}},t\right)  \right) \\
&  =0,
\end{align*}
since $u$ satisfies the heat equation.
\end{proof}

\begin{remark}
Every element of $SL(2;\mathbb{R})$ can be written as $R_{1}DR_{2}$, where
$R_{1}$ and $R_{2}$ in $SO(2)$ and $D$ is positive and diagonal. Thus, the
computations in Examples \ref{MetaRot.ex} and \ref{MetaDiag.ex} allow us to
compute the action of general metaplectic operator in the Segal--Bargmann space.
\end{remark}

%%%%%%%%%%%%%%%%%%%%%%%%%%%%%%%%%%%%%%%%%%%%%%%%%%%%%%%%%%%%%%%%%%%%%%%%%%%%%%%%%%%%%%%%%%%%%%%%%%
%%%% All numerical simulations and figures are in "Poly_IID_coeff.nb"
%%%%%%%%%%%%%%%%%%%%%%%%%%%%%%%%%%%%%%%%%%%%%%%%%%%%%%%%%%%%%%%%%%%%%%%%%%%%%%%%%%%%%%%%%%%%%%%%%%

%\tableofcontents

\section*{Acknowledgments}

BH is supported in part by a grant from the Simons Foundation. CH is supported in part by the MoST grant 111-2115-M-001-011-MY3. JJ and ZK  have been supported by the DFG priority program SPP 2265 \textit{Random Geometric Systems}. ZK has been supported by the German Research Foundation under Germany's Excellence Strategy EXC 2044 -- 390685587, \textit{Mathematics M\"unster:
Dynamics - Geometry - Structure}. 

We thank the referee for a careful reading of the manuscript and useful corrections.

\end{document}